\documentclass[11pt]{article}
\usepackage{amsfonts}
\usepackage{amsmath}
\usepackage{epsfig,amssymb,amsmath,latexsym}
\usepackage[all]{xy}

\newcommand{\be}{\begin{equation}}
\newcommand{\ee}{\end{equation}}
\newcommand{\bea}{\begin{eqnarray}}
\newcommand{\eea}{\end{eqnarray}}
\newcommand{\beas}{\begin{eqnarray*}}
\newcommand{\eeas}{\end{eqnarray*}}

\newtheorem{theorem}{Theorem}[section]
\newtheorem{definition}[theorem]{Definition}
\newtheorem{proposition}[theorem]{Proposition}
\newtheorem{corollary}[theorem]{Corollary}
\newtheorem{lemma}[theorem]{Lemma}
\newtheorem{remark}[theorem]{Remark}
\newtheorem{example}[theorem]{Example}
\newtheorem{examples}[theorem]{Examples}
\newtheorem{foo}[theorem]{Remarks}

\newenvironment{Remark}{\begin{remark}\rm}{\end{remark}}

\newenvironment{proof}{\addvspace{\medskipamount}\par\noindent{\it Proof}.}
{\unskip\nobreak\hfill$\Box$\par\addvspace{\medskipamount}}
\newcommand{\ang}[1]{\left<#1\right>}
\newcommand{\brak}[1]{\left(#1\right)}    % round brackets
\newcommand{\crl}[1]{\left\{#1\right\}}   % curly brackets
\newcommand{\edg}[1]{\left[#1\right]}     % edgy brackets

\newcommand{\p}{\mathbb{P}}

\newcommand{\E}[1]{{\mathbb{E}}\left[#1\right]}

\newcommand{\Emustetig}[1]{\mathbb{E}^{\mathbb{\mu}}_{\F_t}\left[#1\right]}

\newcommand{\N}[1]{\left\|#1 \right\|}     % Norm
\newcommand{\abs}[1]{\left|#1\right|}     % absolute value

\DeclareMathOperator{\esssup}{ess\,sup}

\def\F{\mathcal{F}}

\begin{document}
\title{Existence, minimality and approximation of\\
solutions to BSDEs with convex drivers}

%%%%%%%%%%%%%%%%%%%%%%%%%%%% END CHANGES %%%%%%%%%%%%%%%%%%%%%%%%%%%%%%%%
\author{Patrick Cheridito\thanks{Supported by NSF Grant DMS-0642361}\\
Princeton University\\ Princeton, NJ, USA
\and Mitja Stadje\footnotesize{*}\\
University of Tilburg\\Tilburg, The Netherlands}
\date{May 6, 2011}

\maketitle

\begin{abstract}
We study the existence of solutions to backward stochastic differential equations with drivers
$f(t,W,y,z)$ that are convex in $z$. We assume $f$ to be Lipschitz in $y$ and $W$ but do not make
growth assumptions with respect to $z$. We first show the existence of
a unique solution $(Y,Z)$ with bounded $Z$ if the terminal condition is
Lipschitz in $W$ and that it can be approximated by the solutions
to properly discretized equations. If the terminal condition
is bounded and uniformly continuous in $W$ we show the existence of a
minimal continuous supersolution by uniformly approximating the terminal condition
with Lipschitz terminal conditions. Finally, we prove existence of
a minimal RCLL supersolution for bounded lower semicontinuous
terminal conditions by approximating the terminal condition pointwise
from below with Lipschitz terminal conditions.\\[2mm]
\end{abstract}

\setcounter{equation}{0}
\section{Introduction}

We consider BSDEs (backward stochastic differential equations) of the form
\be \label{bsde}
Y_t = \xi + \int_t^T f(s,W,Y_s,Z_s)ds - \int_t^T Z_s dW_s, \quad 0 \le t \le T,
\ee
with drivers $f$ that are convex in $Z_s$. We assume
$f$ to be Lipschitz-continuous in $W$ and $Y_s$ but only
locally Lipschitz-continuous in $Z_s$. In particular, $f$ can
grow arbitrarily fast in $Z_s$. $(W_t)_{t \in [0,T]}$ is
a $d$-dimensional Brownian motion on a probability space
$(\Omega, {\cal F}, \p)$ and $Z_s dW_s$ is understood as
$\sum_{k=1}^d Z^k_s dW^k_s$. The terminal condition $\xi$ is an
${\cal F}_T$-measurable random variable, where $({\cal F}_t)_{t
\in [0,T]}$ is the augmented filtration generated by $(W_t)_{t \in [0,T]}$.

BSDEs with drivers linear in $(y,z)$ were introduced by Bismut (1973).
Pardoux and Peng (1990) showed that BSDEs with drivers that are Lipschitz
in $(y,z)$ have a unique solution if the terminal condition is square-integrable.
Kobylanski (2000) proved existence and uniqueness of solutions to BSDEs with bounded
terminal conditions and drivers that grow at most
quadratically in $z$. Extensions to unbounded terminal conditions have been
provided by Briand and Hu (2006, 2009) as well as Delbaen et al. (2011).
BSDEs with drivers that are convex and of unrestricted growth in $z$ have
already been studied in Delbaen et al. (2009). In that paper, the Brownian motion
is one-dimensional, the terminal condition is bounded and the driver is of the form
$f(z)$ for a deterministic convex function $f : \mathbb{R} \to \mathbb{R}$ satisfying $f(0) = 0$ and
$\lim_{z \to \pm \infty} f(z)/|z|^2 = \infty$. It is shown in Delbaen et al. (2009) that, depending on
the terminal condition, BSDEs of this form have either no or infinitely many bounded
solutions. Moreover, it is proved that a bounded solution exists if
the terminal condition is of the form $\varphi(X_T)$, where $\varphi : \mathbb{R} \to \mathbb{R}$
is a deterministic bounded continuous function and $X$ a forward process driven by the underlying
Brownian motion. In this special case, BSDEs can be formulated as parabolic PDEs.
Related PDE results have been obtained by Ben-Artzi et al. (2002) and Gilding et al. (2003).

The purpose of this paper is to show the existence and uniqueness of a solution if the driver $f$
depends on $(t,W,y,z)$ and the terminal condition $\xi$ is a possibly unbounded
function of the whole underlying Brownian motion $W_t$, $0 \le t \le T$. However, in view of the
results of Delbaen et al. (2009) it cannot be hoped that solutions exist for
arbitrary terminal conditions or that uniqueness holds without restrictions on the $Z$-process.
Therefore, we first study terminal conditions that are Lipschitz in the underlying
Brownian motion and then approximate more general terminal contitions with Lipschitz ones.
In Theorem \ref{thm1} we show that \eqref{bsde} has a unique solution $(Y,Z)$
with bounded $Z$ if the terminal condition is of the form $\varphi(W)$,
where $\varphi$ is a Lipschitz-continuous function on the space of continuous
functions. Our method of proof is to approximate \eqref{bsde} by discrete-time
equations and show that their solutions converge to a solution of the
continuous-time BSDE. In Theorem \ref{thm2} we prove that for bounded terminal conditions
that can uniformly be approximated by Liptschitz terminal conditions the BSDE
\eqref{bsde} has a bounded continuous supersolution in the sense of Peng (1999) such that
$Z$ is a BMO process. This covers the case of bounded terminal conditions that are uniformly
continuous in the underlying Brownian motion. Theorem \ref{thm3} treats bounded
terminal conditions that are pointwise limits of an increasing sequence of
Lipschitz terminal conditions. In this case we show that the BSDE \eqref{bsde} has a bounded
RCLL supersolution such that $Z$ is BMO. This gives the existence of a RCLL supersolution for
bounded terminal conditions that are lower semicontinuous in the underlying
Brownian motion. If the driver is monotone in $y$, we are also able to show that
the BSDE \eqref{bsde} satisfies a one-sided comparison principle, from which
we deduce that the supersolutions constructed in Theorems \ref{thm2} and \ref{thm3}
are minimal.

The structure of the paper is as follows: In Section 2 we introduce
the notation and state our main results. In Section 3 we prove results on BS$\Delta$Es
(backward stochastic difference equations) that are needed in the proof of Theorem \ref{thm1}
given in Section 4. In Section 5 we use convex duality to show
comparison results. In Section 6 we give the proofs of Theorem \ref{thm2} and \ref{thm3}.
In the Appendix we show that a convergence result of Briand et al. (2002) which
we need in the proof of Theorem 2.3 still holds in our setting.

\setcounter{equation}{0}
\section{Notation and statement of results}

Let $(\Omega, {\cal F}, \p)$ be a probability space carrying a $d$-dimensional
Brownian motion $(W_t)_{0 \le t \le T}$. As usual, we identify random variables
that agree almost surely and understand equalities
as well as inequalities between them in the almost sure sense. Fix $T \in (0,\infty)$
and denote by $C^d[0,T]$ the space of
all continuous functions $w : [0,T] \to \mathbb{R}^d$. Let $({\cal E}_t)$ be the
filtration on $C^d[0,T]$ generated by the coordinate process and ${\cal P}$ the
predictable sigma-algebra on $[0,T] \times C^d[0,T]$. We call a function
$$
f : [0,T] \times C^d[0,T] \times \mathbb{R} \times \mathbb{R}^d \to \mathbb{R}
$$
a driver if it is $\mathcal{P} \otimes \mathcal{B}(\mathbb{R}) \otimes \mathcal{B}(\mathbb{R}^{d})$-measurable.
We always assume $f$ in equation \eqref{bsde} to be a driver and the terminal condition $\xi$
an ${\cal F}_T$-measurable random variable.
We call a stochastic process RCLL if almost all of its paths are right-continuous and have left
limits. We call a stochastic process $(A_t)$ increasing if
$A_s \le A_t$ for $s \le t$. Similarly, we say a function $f : \mathbb{R} \to \mathbb{R}$ is
increasing (decreasing) if $f(x) \le (\ge) f(y)$ for $x \le y$.

\begin{definition}
\label{super} A solution of the BSDE \eqref{bsde} consists of a
pair $(Y_t, Z_t)_{0 \le t \le T}$ of predictable processes with
values in $\mathbb{R} \times \mathbb{R}^d$ such that \be
\label{integrable} \int_0^T |f(s,W,Y_s,W_s)| ds < \infty, \quad
\int_0^T |Z_s|^2ds < \infty \ee and
$$
Y_t = \xi + \int_t^T f(s,W,Y_s,Z_s)ds -\int_t^T Z_s dW_s \quad \mbox{for all } t \in [0,T].
$$
A supersolution of the BSDE \eqref{bsde} consists of a triple
$(Y_t, Z_t, A_t)_{0 \le t \le T}$ of predictable processes with values in
$\mathbb{R} \times \mathbb{R}^d \times \mathbb{R}$ such that
$(Y_t)$ is RCLL, $(A_t)$ starts at $0$ and is increasing RCLL, \eqref{integrable} holds and
$$ Y_t = \xi + \int_t^T f(s,W,Y_s,Z_s)ds -\int_t^T
Z_s dW_s + A_T - A_t \quad \mbox{for all } t \in [0,T].
$$
\end{definition}

\begin{definition}
\label{comparison} We say a supersolution $(Y_t,Z_t,A_t)$ of
the BSDE \eqref{bsde} satisfies bounded comparison from above
if for every supersolution $(Y'_t,Z'_t,A'_t)$ of \eqref{bsde}
with driver $f' \ge f$ and terminal condition $\xi' \ge \xi$
such that $Y'$ is bounded, one has $Y'_t \ge Y_t$ for all $t$.
\end{definition}

\begin{Remark}
If $(Y_t,Z_t,A_t)$ is a supersolution of the BSDE \eqref{bsde}
such that $Y$ is bounded and satisfies bounded comparison from
above, one has $Y'_t \ge Y_t$, $0 \le t \le T$, for every other
supersolution $(Y'_t,Z'_t,A'_t)$ of \eqref{bsde} such that $Y'$
is bounded. So $(Y_t,Z_t,A_t)$ is the minimal bounded
supersolution of \eqref{bsde}. If in addition, $ A\equiv 0$, $(Y_t,Z_t)$ is
the minimal bounded solution.
\end{Remark}

Denote by $|.|$ the Euclidean norm on $\mathbb{R}^d$. For most of our results we
need the driver to satisfy some or all of the following properties:
\begin{itemize}
\item[(f1)] $f(t,w,y,z)$ is convex in $z$
\item[(f2)] $\sup_{t,w} |f(t,w,0,0) |<\infty$
\item[(f3)] There exists a constant $K \in \mathbb{R}_+$ such that
$$
|f(t,w_{1},y_1,z)-f(t,w_{2},y_2,z)| \le K \big(\sup_{0 \le s \le t}
|w_1(s) -w_2(s)|+|y_1-y_2|\big)
$$
for all $t, w_1, w_2, y_1, y_2, z$.
\item[(f4)]
For every $a \in \mathbb{R}_+$ there exists a $b \in \mathbb{R}_+$ such that
$$
|f(t,w,y,z_1) - f(t,w,y,z_2)| \le b |z_1- z_2|
$$
for all $t,w,y$ and $z_1,z_2 \in\mathbb{R}^d$ with $|z_1| \vee |z_2| \le a$.
\item[(f5)] $\inf_{w \in C^d[0,T],y \in [-c,c],\,\,z\in\mathbb{R}^d,t\in[0,T]}f(t,w,y,z)>-\infty
\quad \mbox{for all } c \in \mathbb{R}_+$,\\[2mm]
\end{itemize}
It can be shown that it follows from (f3) and (f4) that $f$ is
$\mathcal{P} \otimes \mathcal{B}(\mathbb{R}) \otimes \mathcal{B}(\mathbb{R}^{d})$-measurable
and therefore, a driver.

Our first result shows that the BSDE \eqref{bsde} has a unique solution
such that $Z$ is bounded when $f$ satisfies (f1)--(f4) and the terminal
condition is Lipschitz-continuous in the underlying Brownian motion $W$.
We prove it by discretizing equation \eqref{bsde} in time and then passing to
the continuous-time limit. To do that
we approximate $W$ by a sequence $W^N$, $N
\in \mathbb{N}$, of $d$-dimensional square-integrable
martingales starting at $0$ with independent increments
satisfying the following conditions:
\begin{itemize}
\item[(W1)]
For every $N \in \mathbb{N}$ there exists a
finite sequence  $0=t^N_0 < t^N_1< t^N_2\cdots < t^N_{i_N}=T$ such that
$$
\lim_{N\rightarrow\infty} \max_i |t^N_{i+1} - t^N_i| = 0
$$
and $W^N_t$ is constant on the intervals $[t^N_i ,t^N_{i+1})$.
\item[(W2)]
$$
\lim_{N \to \infty} \E{\sup_{0\le t\le T} |W^N_t-W_t|^2} = 0.
$$
\item[(W3)]
For all $N$ and $i$, $\Delta W^N_{t^N_i}$ takes only finitely
many different values.
\item[(W4)]
For all $N$, $i$ and $k \neq l$,
$$
\E{\Delta W^{N,k}_{t^N_i} \Delta W^{N,l}_{t^N_i}} = 0 \quad
\mbox{and} \quad \Delta \ang{W^{N,k}}_{t^N_i} = \Delta \ang{W^{N,l}}_{t^N_i}=\Delta t^N_i > 0 .
$$
\item[(W5)]
$$\sup_{N,i,k} \frac{\N{\Delta W^{N,k}_{t^N_i}}_{\infty}}{\sqrt{\Delta t^N_i}}
< \infty.
$$
\end{itemize}
One can, for instance, set $t^N_i = iT/N$, $i =0, \dots , N$
and let the $W^N$ be $d$-dimensional Bernoulli random walks
with increments $\pm \sqrt{T/N}$, that is, the increments
$W^{N,k}_{t^N_i} - W^{N,k}_{t^N_{i-1}}$, $i = 1, \dots, N$, $k
= 1, \dots d$, are independent and have distribution $\p
\edg{W^{N,k}_{t^N_i} - W^{N,k}_{t^N_{i-1}} = \pm \sqrt{T/N}} =
1/2$ (see Cheridito and Stadje (2009) for details on how to
construct $d$-dimensional Bernoulli random walks on the same
probability space as $W$ such that on has the convergence of (W2)).

We set $\ang{W^N}_t := \ang{W^{N,1}}_t = \dots = \ang{W^{N,d}}_t = t^N_i$
for $t^N_i \le t < t^N_{i+1}$. Let $({\cal F}^N_t)$ be the filtration generated by $W^N$.
To define the approximating BS$\Delta$Es, we construct two continuous approximations
to $W^N$. The process
\be \label{continuousbar}
\bar{W}^N_t = W^N_{t^N_{i-1}} + \frac{t - t^N_{i-1}}{t^N_i - t^N_{i-1}}
(W^N_{t^N_i} - W^N_{t^N_{i-1}}) \quad \mbox{for } t^N_{i-1} \le t \le t^N_i
\ee
is continuous but not adapted to $({\cal F}^N_t)$. To make it
$({\cal F}^N_t)$-adapted, we shift it by $h^N := \sup_i |t^N_i - t^N_{i-1}|$ and define
\be \label{continuous}
\hat{W}^N_t = \left\{ \begin{array}{cl}
0 & \mbox{ for } 0 \le t \le h^N\\
\bar{W}^N_{t - h^N} & \mbox{ for } h^N \le t \le T.
\end{array} \right.
\ee

Introduce the left-continuous, piecewise constant process
$\hat{f}^N$ on $C^d[0,T]$ by $\hat{f}^N(0,w,y,z) :=f(0,w,y,z)$ and
\be
\label{deffn}
\hat{f}^N(t,w,y,z):=\frac{\int_{t^N_i}^{t^N_{i+1}}f(s,w,y,z)ds}{\Delta
t^N_{i+1}} \quad \mbox{for }t^N_i<t\leq t^N_{i+1}.
\ee
Since the approximating processes $W^N$ do in general not have the
predictable representation property, solutions to the discretized equations
involve orthogonal martingales. More, precisely, a solution to the
$N$-th BS$\Delta$E (backward stochastic difference equation)
corresponding to an ${\cal F}^N_T$-measurable terminal condition $\xi^N$ consists of a
triple of $({\cal F}^N_t)$-adapted processes
$(Y^N_t,Z^N_t,M^N_t)$ taking values in $\mathbb{R} \times
\mathbb{R}^d \times \mathbb{R}$ such that $(Y^N_t)$ is constant
on the intervals $[t^N_i, t^N_{i+1})$, $(Z^N_t)$ is constant on
the intervals $(t^N_i, t^N_{i+1}]$, $(M^N_t)$ is a martingale
starting at $0$ and orthogonal to $(W^N_t)$ that is constant on
the intervals $[t^N_i, t^N_{i+1})$ and
\be \label{Nbsde}
Y^N_t = \xi^N + \int_{(t,T]} \hat{f}^N(s,\hat{W}^N,Y^N_{s-},
Z^N_s) d \ang{W^N}_s - \int_{(t,T]} Z^N_s d W^N_s - (M^N_T - M^N_t), \quad t \in [0,T].
\ee
Since the process $(W^N_t)$ is piece-wise constant, it is
completely determined by the finite sequence $(W^N_{t^N_1},\dots, W^N_T)$,
and equation \eqref{Nbsde} can be written as
\bea
\label{Nbsde1} Y^N_{t^N_i} &=& Y^N_{t^N_{i+1}} +
f^N(t^N_{i+1},W^N,Y^N_{t^N_{i}}, Z^N_{t^N_{i+1}}) \Delta
t^N_{i+1}
- Z^N_{t^N_{i+1}} \Delta W^N_{t^N_{i+1}} -\Delta M^N_{t^N_{i+1}}\\
\label{Nbsde2}
Y^N_T &=& \xi^N,
\eea
for functions
$$
f^N : \crl{t^N_1, \dots, T} \times \mathbb{R}^{d \times i_N}
\times \mathbb{R} \times \mathbb{R}^d\to\mathbb{R}.
$$
If $f$ satisfies (f1)--(f4), $f^N$ has the following properties:
\begin{itemize}
\item[(f1')] $f^N(t^N_i,w,y,z)$ is convex in $z$
\item[(f2')] $\sup_{N,i} |f^N(t^N_i,W^N,0,0)|<\infty$
\item[(f3')]
There exists a constant $K \in \mathbb{R}_+$ such that
$$
|f^N(t^N_i,w_1,y_1,z) - f^N(t^N_i,w_2,y_2,z)| \le K \big(\sup_{j \le i-1}
|w_1(t^N_j) -w_2(t^N_j)|+|y_1-y_2|\big)
$$
for all $N,i, w_1, w_2, y_1, y_2, z$.
\item[(f4')]
For every $a \in \mathbb{R}_+$ there exists a $b \in \mathbb{R}_+$ such that
$$
|f^N(t^N_i,w,y,z_1) - f^N(t^N_i,w,y,z_2)| \le b |z_1- z_2|
$$
for all $N,i,w,y$ and $z_1,z_2 \in\mathbb{R}^d$ satisfying $|z_1| \vee |z_2| \le a$.
\end{itemize}

We endow $C^d[0,T]$ with the supremum norm $\N{w}_{\infty} := \sup_{0 \le t \le T} |w(t)|$.
Our first result assumes that
the terminal condition is of the form $\xi = \varphi(W)$ for a
Lipschitz-continuous function $\varphi : C^d[0,T] \to \mathbb{R}$,
that is, there exists a constant $L \in \mathbb{R}$ such that
$|\varphi(w_1) - \varphi(w_2)| \le L \N{w_1 - w_2}_{\infty}$ for all
$w_1,w_2 \in C^d[0,T]$.

\begin{theorem} \label{thm1}
Assume $f$ satisfies {\rm (f1)--(f4)} and $\xi$ is of the form
$\xi = \varphi(W)$ for a Lipschitz-continuous function $\varphi :
C^d[0,T] \to \mathbb{R}$. Then the BSDE \eqref{bsde} has a
unique solution $(Y,Z)$ such that $Z$ is bounded. Moreover, if
$\xi^N = \varphi(\hat{W}^N)$, then for $N$ large enough, there
exist unique solutions $(Y^N,Z^N,M^N)$ to the corresponding
BS$\Delta$Es \eqref{Nbsde} and \be \label{conver1} \sup_t
\brak{|Y^N_t-Y_t| + |\int_0^t Z^N_s dW^N_s - \int_0^t Z_s
dW_s|+ |M^N_t|} \to 0 \quad \mbox{in } L^2 \ee as well as \be
\label{conver2} \sup_t \brak{\sum_{k=1}^d \abs{\int_0^t
Z^{N,k}_s d \ang{W^N}_s - \int_0^t Z^k_s ds}^2 + \abs{\int_0^t
|Z^N_s|^2 d \ang{W^N}_s - \int_0^t |Z_s|^2 ds}} \to 0 \quad
\mbox{in } L^1. \ee If $(Y',Z')$ is the solution with bounded
$Z'$ of the BSDE \eqref{bsde} corresponding to a driver $f' \ge
f$ satisfying {\rm (f1)--(f4)} and terminal condition $\xi' \ge
\xi$ of the form $\xi' = \varphi'(W)$ for a
Lipschitz-continuous function $\varphi' : C^d[0,T] \to
\mathbb{R}$, then $Y'_t \ge Y_t$ for all $t \in [0,T]$. In
particular, if $\varphi$ is bounded, then $Y$ is bounded as well.
\end{theorem}

Next, we consider terminal conditions that can be uniformly approximated
by Lipschitz-continuous terminal conditions.
We call a $d$-dimensional $({\cal F}_t)$-predictable
process $(\mu_t)_{t\in[0,T]}$ BMO if there exists a
constant $C \in \mathbb{R}_+$ such that
$$
\E{\int_{\tau}^T |\mu_s|^2 ds| \F_\tau} \le C
$$
for all stopping times $\tau$ taking values in $[0,T]$.
By choosing $\tau = 0$, one obtains that a BMO process $\mu$ satisfies
$\E{\int_0^T |\mu_s|^2 ds} < \infty$.

\begin{theorem} \label{thm2}
Assume $f$ satisfies {\rm (f1)--(f5)} and $\varphi^n : C^d[0,T] \to
\mathbb{R}$ is a sequence of bounded Lipschitz-continuous functions such that
$\N{\varphi^n - \varphi}_{\infty} \to 0$ for a bounded function $\varphi : C^d[0,T] \to \mathbb{R}$.
Denote by $(Y^n,Z^n)$ the solution of the BSDE \eqref{bsde}
with terminal condition $\xi^n = \varphi^n(W)$ such that $Z^n$ is bounded.
Then
$$
\N{\sup_t \big|Y^n_t - Y_t\big|}_{L^{\infty}} \to 0 \quad \mbox{and} \quad
\E{\sqrt{\int_0^T |Z^n_s - Z_s|^2 ds}} \to 0,
$$
where $(Y,Z,A)$ is a supersolution of \eqref{bsde} such that $Y$ is
bounded and continuous and $Z$ is a ${\rm BMO}$ process. If moreover, $f$
is increasing or decreasing in $y$, then $(Y,Z,A)$ satisfies bounded comparison from above
and hence, is the minimal bounded supersolution of the BSDE \eqref{bsde}.
\end{theorem}

Since the uniform limits of bounded Lipschitz-continuous
functions on $C^d[0,T]$ are all the bounded uniformly
continuous functions on $C^d[0,T]$, the following corollary is
an immediate consequence of Theorem \ref{thm2}:

\begin{corollary}
If $f$ satisfies {\rm (f1)--(f5)} and the terminal condition is
of the form $\xi = \varphi(W)$ for a bounded uniformly
continuous function $\varphi : C^d[0,T] \to \mathbb{R}$, then
the BSDE \eqref{bsde} has a supersolution $(Y,Z,A)$ such that
$Y$ is bounded and continuous and $Z$ is a {\rm BMO} process. If in addition,
$f$ is increasing or decreasing in $y$, then \eqref{bsde} has a minimal bounded
supersolution $(Y,Z,A)$. It satisfying bounded comparison from above,
$Y$ is continuous and $Z$ is a {\rm BMO} process.
\end{corollary}

The next result is about terminal conditions that can be approximated pointwise from
below by Lipschitz-continuous terminal conditions:

\begin{theorem} \label{thm3}
Assume $f$ satisfies {\rm (f1)--(f5)} and is increasing in $y$.
Let $\varphi^n : C^d[0,T] \to \mathbb{R}$ be a sequence of bounded
Lipschitz-continuous functions such that $\varphi^n \uparrow \varphi$
pointwise for a bounded function $\varphi : C^d[0,T] \to \mathbb{R}$.
Denote by $(Y^n,Z^n)$ the solution of the BSDE
\eqref{bsde} corresponding to the terminal condition $\xi^n = \varphi^n(W)$ such that $Z^n$ is bounded.
Then $Y^n_t \uparrow Y_t$ a.s. for all $t$,
where $(Y,Z,A)$ is a supersolution of \eqref{bsde} satisfying
bounded comparison from above such that $Y$ is bounded and $Z$ is a {\rm BMO} process.
\end{theorem}

Note that every bounded function $\varphi : C^d[0,T] \to \mathbb{R}$
that is the pointwise limit of an increasing
sequence of bounded Lipschitz-continuous functions $\varphi^n :
C^d[0,T] \to \mathbb{R}$ is lower semicontinuous. On the other
hand, for every bounded lower semicontinuous function $\varphi
: C^d[0,T] \to \mathbb{R}$, the functions
$$
\varphi^n(w) := \inf_{v \in C^d[0,T]} \varphi(v) + n \N{v - w}_{\infty}
$$
are bounded Lipschitz-continuous and increase pointwise to
$\varphi$. This gives the following corollary to Theorem \ref{thm3}:

\begin{corollary}
If $f$ satisfies {\rm (f1)--(f5)} and is increasing in $y$, then for every bounded lower
semicontinuous function $\varphi : C^d[0,T] \to \mathbb{R}$,
the BSDE \eqref{bsde} with terminal condition $\xi = \varphi(W)$ has a minimal bounded
supersolution $(Y,Z,A)$. It satisfies bounded comparison from above
and $Z$ is a {\rm BMO} process.
\end{corollary}

\setcounter{equation}{0}
\section{Solutions of BS$\Delta$Es and their properties}

\begin{lemma} \label{lemmaYZM}
If the BS$\Delta$E \eqref{Nbsde1}--\eqref{Nbsde2} has a solution $(Y^N,Z^N,M^N)$, then
\bea \label{Yformula}
Y^N_{t^N_i} &-& f^N(t^N_{i+1},W^N,Y^N_{t^N_i},Z^N_{t^N_{i+1}})
\Delta t^N_{i+1}=\E{Y^N_{t^N_{i+1}}|\F^N_{t^N_i}}\\
\label{Zformula} Z^{N,k}_{t^N_{i+1}} &=&
\frac{\E{Y^N_{t^N_{i+1}}
\Delta W^{N,k}_{t^N_{i+1}}|\F^N_{t^N_i}}}{\Delta t^N_{i+1}}\\
\label{Mformula}
\Delta M^N_{t^N_{i+1}} &=& Y^N_{t^N_{i+1}} - \E{Y^N_{t^N_{i+1}}|\F^N_{t^N_i}}
- Z^N_{t^N_{i+1}} \Delta W^N_{t^N_{i+1}}
\eea
for all $i \le i_N-1$.
\end{lemma}

\begin{proof}
\eqref{Yformula} follows from equation \eqref{Nbsde1} by
taking conditional expectation with respect to ${\cal F}^N_{t^N_i}$.
\eqref{Zformula} is obtained by first multiplying \eqref{Nbsde1} with
$\Delta W^{N,k}_{t^N_{i+1}}$ and then taking conditional expectation with respect to
${\cal F}^N_{t^N_i}$. \eqref{Mformula} is a consequence of \eqref{Nbsde1} and
\eqref{Yformula}.
\end{proof}

By condition (W1), there exists $N_0 \in \mathbb{N}$ such that
$\max_i \Delta t^N_i < 1/K$ for all $N \ge N_0$. So it follows from the following
proposition that for large enough $N$, the BS$\Delta$E \eqref{Nbsde1}--\eqref{Nbsde2} has
a unique solution for every terminal condition.

\begin{proposition} \label{propex}
If $\max_i \Delta t^N_{i} < 1/K$, the $N$-th
BS$\Delta$E has for every terminal condition $\xi^N$ a
unique solution $(Y^N, Z^N, M^N)$.
\end{proposition}

\begin{proof}
We show the proposition by backwards induction. One must have
$Y^N_T = \xi^N$, and if $Y^N_{t^N_{i+1}}$ is given, the only possible choice for
$Z^N_{t^N_{i+1}}$ is \eqref{Zformula}.
Since $\Delta t^N_i K < 1$, one obtains from (f3) that
for every possible realization $(w(t^N_1), \dots, w(T)) \in
\mathbb{R}^{d \times i_N}$ of $(W^N_{t^N_1}, \dots, W^N_T)$, there exists a unique
$y \in \mathbb{R}$ such that
$$
y - f^N(t^N_{i+1},w,y,Z^N_{t^N_{i+1}}) \Delta t^N_{i+1} = \E{Y^N_{t^N_{i+1}}|\F^N_{t^N_i}}.
$$
This gives an ${\cal F}_{t^N_i}$-measurable $Y^N_{t^N_i}$ random variable
satisfying \eqref{Yformula}. Finally, one defines
$(M^N_t)$ through $M^N_0 = 0$ and \eqref{Mformula}. Then
$(Y^N_t,Z^N_t,M^N_t)$ is the unique solution of the BS$\Delta$E \eqref{Nbsde1}--\eqref{Nbsde2}.
\end{proof}

Let us denote by $g^N$ the convex conjugate of $f^N$ with respect to $z$, given by
$$
g^N(t,w,y,\mu) = \sup_{z \in \mathbb{R}^d} \crl{z \mu - f^N(t,w,y,z)}, \quad \mu \in \mathbb{R}^d.
$$
$g^N$ is a mapping from $[0,T] \times \mathbb{R}^{d \times i_N}\times \mathbb{R} \times \mathbb{R}^d$
to $\mathbb{R} \cup \crl{\infty}$, which inherits condition (f3') from $f^N$, that is,
$$
|g^N(t_i,w_1,y_1,\mu) - g^N(t_i,w_2,y_2,\mu)| \le K \big(\sup_{j \le i-1}
|w_1(t^N_j) -w_2(t^N_j)|+|y_1-y_2|\big)
$$

Our next goal is to obtain an implicit convex dual representation of
$Y^N_t$ in terms of $g^N$. We need the following notation: Let
$(\mu_t)$ be an $({\cal F}^N_t)$-adapted $\mathbb{R}^d$-valued process constant on the intervals
$(t^N_i, t^N_{i+1}]$ such that
\be \label{mucond}
\mu_{t_i} \Delta W^N_{t^N_i} > -1 \quad \mbox{for all } i.
\ee
Then
\begin{equation}
\label{Pmu}
\frac{d \p^{\mu}}{d \p} = \prod_{i=1}^{i_N} (1 + \mu_{t^N_i} \Delta W^N_{t^N_i})
\end{equation}
defines a probability measure $\p^{\mu}$ equivalent to $\p$ under which the processes
$$
W^{N, \mu, k}_{t^N_i} := W^{N,k}_{t^N_i} - \sum_{j=1}^i \mu^k_{t^N_j}
\Delta t^N_j, \quad i = 1, \dots, i_N, \quad k = 1, \dots, d,
$$
are $({\cal F}^N_t)$-martingales. Note that $M^N$ is still a martingale under $\p^{\mu}$.

\begin{lemma} \label{lemmadualY}
For every constant $C>0$, there exists an $N_0 \in \mathbb{N}$ such that for all
$N \ge N_0$ the following holds: If there is an $i \le i_N-1$ such that the
$N$-th BS$\Delta$E has a solution $(Y^N, Z^N, M^N)$ satisfying
$|Z^{N}_{t^N_j}| \le C$ for all $j \ge i+1$, then
\be \label{dualq}
Y^N_{t^N_i} = \sup_{\mu} \mathbb{E}^{\mu}
\edg{\xi^N - \sum_{j=i+1}^{i_N} g^N(t^N_j,W^N,Y^N_{t^N_{j-1}},\mu_{t^N_j}) \Delta t^N_j \mid {\cal F}^N_{t^N_i}},
\ee
where the supremum is taken over all
$({\cal F}^N_t)$-adapted $\mathbb{R}^d$-valued processes $(\mu_t)$ that are
constant on the intervals $(t^N_{j}, t^N_{j+1}]$  and satisfy \eqref{mucond}.
Furthermore, the supremum is attained for some process $(\mu^*_t)$.
% that is
%uniformly bounded by any constant $D$ satisfying
%\be \label{Lipf}
%|f^N(t^N_i,w,y,z_1) - f^N(t^N_i,w,y,z_2)| \le D |z_1- z_2|
%\ee
%for all $N,i,w,y$, and $z_1,z_2 \in\mathbb{R}^d$ with $|z_1| \vee |z_2| \le 2C$.
\end{lemma}

\begin{proof}
First assume $(Y^N,Z^N,M^N)$ is a solution of the $N$-the BS$\Delta$E and
$\mu$ is an $({\cal F}^N_t)$-adapted $\mathbb{R}^d$-valued process that is constant on the intervals
$(t^N_j, t^N_{j+1}]$ and satisfies \eqref{mucond}. Since
$(W^{N,\mu,k}_t)$ is for all $k = 1, \dots, d$, a martingale under $\p^{\mu}$, one obtains
for every $i \le i_N-1$,
\bea
Y^N_{t^N_i} &=& \mathbb{E}^{\mu}
\Big[\xi^N + \sum_{j=i+1}^{i_N} f^N(t^N_j,W^N,Y^N_{t^N_{j-1}},Z^N_{t^N_j}) \Delta  t^N_j
\nonumber\\
&& -\sum_{j=i+1}^{i_N} Z^N_{t^N_j} \Delta W^N_{t^N_j} - (M^N_T-M^N_{t^N_i}) \mid {\cal F}^N_{t^N_i} \Big]
\nonumber\\
&=& \mathbb{E}^{\mu} \Big[\xi^N + \sum_{j=i+1}^{i_N}
\Big(f^N(t^N_j,W^N, Y^N_{t^N_{j-1}}, Z^N_{t^N_{j}}) -Z^N_{t^N_j}
\mu_{t^N_j} \Big) \Delta  t^N_j
\nonumber\\
&& - \sum_{j=i+1}^{i_N} Z^{N,x}_{t^N_j} \Big(\Delta W^N_{t^N_{j}} -\mu_{t^N_j} \Delta t^N_j \Big)
\mid {\cal F}^N_{t^N_i} \Big]
\nonumber\\
&=& \mathbb{E}^{\mu} \Big[\xi^N + \sum_{j=i+1}^{i_N} \Big(f^N(t^N_{j},W^N,Y^N_{t^N_{j-1}},
Z^N_{t^N_j})- Z^N_{t^N_j} \mu_{t^N_j} \Big) \Delta t^N_j \mid {\cal F}^N_{t^N_i} \Big] \nonumber\\
&\ge& \mathbb{E}^{\mu} \Big[\xi^N - \sum_{j=i+1}^{i_N} g^N(t^N_{j},W^N, Y^N_{t^N_{j-1}},
\mu_{t^N_j}) \Delta t^N_j \mid {\cal F}^N_{t^N_i} \Big] \label{discdualinequ}.
\eea
Now let $C > 0$. By condition (f4'), there exists a constant $b \in \mathbb{R}_+$ such that
\be \label{Lipf}
|f^N(t^N_i,w,y,z_1) - f^N(t^N_i,w,y,z_2)| \le b |z_1- z_2|
\ee
for all $N,i,w,y$, and $z_1,z_2 \in\mathbb{R}^d$ with $|z_1| \vee |z_2| \le 2C$.
Due to (W5) there exists a $D \in \mathbb{R}_+$ such that
$$
\sup_{N,i,k} \frac{\N{\Delta W^{N,k}_{t^N_i}}_{\infty}}{\sqrt{\Delta  t^N_i}} \le D,
$$
and by (W1), there is an $N_0 \in \mathbb{N}$ such that
\begin{equation}
\label{n02}
\sup_i b \sqrt{d} D \sqrt{\Delta t^N_{i}} < 1 \quad \mbox{for all } N \ge N_0.
\end{equation}
Fix $N \ge N_0$ and $i \in \{0,\cdots,i_N-1\}$. Assume
$(Y^N,Z^N,M^N)$ is a solution of the $N$-th BS$\Delta$E such that $|Z^N_{t^N_j}| \le C$ for all $j \ge i+1$.
To see that inequality \eqref{discdualinequ} is actually an equality for
some  process $(\mu^*_t)$, note that the subdifferential $\partial f(t,w,y,z)$
of $f$ with respect to $z$ is non-empty for all $(t,w,y,z)$.
For every $j \ge i+1$, the filtration ${\cal F}^N_{t^N_{j-1}}$ has only finitely many atoms
$B_1, \dots ,B_m$. On every atom $B_l$ choose a vector $z_l \in \partial
f(t^N_j,W^N, Y^N_{t^N_{j-1}}, Z^N_{t^N_j})$. Set $\mu^*_{t^N_j} := z_l$
for $t \in (t^N_{j-1}, t^N_j]$ and $\omega \in B_l$ and $\mu^*_t := 0$ for $t \le t^N_i$.
Then $(\mu^*_t)$ is an $({\cal F}^N_t)$-adapted $\mathbb{R}^d$-valued
process constant on the intervals $(t^N_i, t^N_{i+1}]$ such that
$$
Z^N_{t^N_j} \mu^*_{t^N_j} = f^N(t^N_j,W^N, Y^N_{t^N_{j-1}},Z^N_{t^N_j})
+ g^N(t^N_j, W^N, W^N,Y^N_{t^{N}_{j-1}},\mu^*_{t^N_j}) \quad \mbox{for all } j \ge i+1.
$$
It remains to show that $\mu^*$ satisfies condition \eqref{mucond}.
Then $\p^{\mu^*}$ is a probability measure equivalent to $\p$ and for $\mu = \mu^*$, the inequality in
\eqref{discdualinequ} becomes an equality. But since $|Z^N_{t^N_j}| \le C$ for all $j \ge i+1$,
it follows from \eqref{Lipf} that $|\mu^*_{t^N_{j}}| \le b$, and one obtains
$$
|\mu^*_{t^N_j} \Delta W^N_{t^N_j}| \le b \sqrt{d} D \sqrt{\Delta t^N_j} < 1
\quad \mbox{for all } j \ge i+1.
$$
\end{proof}

\setcounter{equation}{0}
\section{Proof of Theorem \ref{thm1}}

We need the following discrete-time version of Gronwall's lemma:

\begin{lemma} \label{lemmaGronwall}
For every $B \in \mathbb{R}_+$ there exists $N_0 \in \mathbb{N}$
such that for all $N \ge N_0$ the following holds:
If $(X^N_t)_{t \in [0,T]}$ is a stochastic process that is
constant on the intervals $[t^N_{i-1}, t^N_i)$ and satisfies
$$
|X^N_T| \le A \quad \mbox{as well as} \quad
|X^N_{t^N_i}| \le A + B \sum_{j=i+1}^{i_N} |X^N_{t^N_{j-1}}| \Delta t^N_{j},
\quad i \le i_N -1 \quad \mbox{for some } A \in \mathbb{R}_+,
$$
then
$$
|X^N_t| \le 2 A \exp\{B(T-t)\}, \quad \mbox{for all } t \in [0,T].
$$
\end{lemma}

\begin{proof}
If $N$ is so large that $B \max_i \Delta t^N_{i} < 1$, then the unique
process that is constant on the intervals $[t^N_{i-1}, t^N_i)$ and solves the deterministic backward equation
$$
\hat{X}^N_T = A, \quad \hat{X}^N_{t^N_i} = A +
B \sum_{j = i+1}^{i_N}\hat{X}^N_{t^N_{j-1}} \Delta t^N_{j},
\quad i \le i_N -1,
$$
is given by
$$
\hat{X}^N_T = A \quad \mbox{and} \quad
\hat{X}^N_t = A \prod_{j\,:\,t^N_j > t} (1 - B \Delta t^N_{j})^{-1}
\quad \mbox{for } t < T.
$$
Since $\prod_{j\,:\,t^N_j > t} (1 - B \Delta t^N_{j})^{-1}$ is
converging uniformly in $t$ to $\exp(B(T-t))$, there exists $N_0
\in \mathbb{N}$ such that $B \max_i \Delta t^N_{i} < 1$ and
$\prod_{j\,:\,t^N_j > t} (1 - B \Delta t^N_{j})^{-1}\leq 2\exp(B(T-t))$
for all $N \ge N_0$ and $t \in [0,T]$. Therefore, $\hat{X}^N_t
\le 2 A \exp(B(T-t))$ for all $N \ge N_0$ and $t \in [0,T]$. It
remains to show that $|X^N_t| \le \hat{X}^N_t$.
But this follows by backwards induction from
$$
|X^N_{t^N_i}| \le \frac{A + B\sum_{j=i+2}^{i_N} |X^N_{t^N_{j-1}}|
\Delta t^N_{j}}{1 - B \Delta t^N_{i+1}}
\le \frac{A + B\sum_{j=i+2}^{i_N} \hat{X}^N_{t^N_{j-1}}
\Delta t^N_{j}}{1 - B \Delta t^N_{i+1}}
= \hat{X}^N_{t^N_i}.
$$
\end{proof}

\begin{lemma} \label{lemmaZbound}
Assume all $\xi^N$ are of the form $\xi^N = \varphi(\hat{W}^N)$ for a function
$\varphi : C^d[0,T] \to \mathbb{R}$ for which there exists a constant $L \in \mathbb{R}_+$ such that
\be \label{phiLip}
|\varphi(w_1) - \varphi(w_2)| \le L \N{w_1 - w_2}_{\infty}
\ee
for all $w_1,w_2 \in C^d[0,T]$. Then there exists an $N_0 \in \mathbb{N}$ such that
for $N \ge N_0$, every solution $(Y^N, Z^N, M^N)$ of the $N$-th BS$\Delta$E satisfies
\be \label{boundedZ}
\sup_{0 \le t \le T} |Z^N_t| \le 2 \sqrt{d} (L+KT) \exp(KT).
\ee
\end{lemma}

\begin{proof}
Choose $N_0 \in \mathbb{N}$ so large that the statement of Lemma \ref{lemmadualY} holds for
$C = 2 \sqrt{d} (L+KT) \exp(KT)$ and the statement of Lemma \ref{lemmaGronwall} holds for $B = K$.
Assume $(Y^N,Z^N,M^N)$ is a solution of the $N$-th BS$\Delta$E for some $N \ge N_0$.
We prove \eqref{boundedZ} by backwards induction.
Fix $i \in \crl{1,\dots,i_N}$ and if $i \le i_N-1$, assume
$$
|Z^N_{t^N_j}| \le 2 \sqrt{d} (L+KT) \exp(KT) \quad \mbox{for all } j \ge i+1.
$$
There exist functions $\varphi^N : \mathbb{R}^{d \times i_N} \to \mathbb{R}$
such that $\varphi^N(W^N_{t^N_1},\ldots,W^N_T) = \varphi(\hat{W}^N)$ and
\be \label{phinlip}
|\varphi^N(w_1,\ldots,w_{i_N}) - \varphi^N(w'_1,\ldots,w'_{i_N})| \le L \sup_{i=1,\ldots,i_N} |w_i - w'_i|
\ee
for all $w_1, \ldots, w_{i_N}, w'_1, \ldots, w'_{i_N} \in \mathbb{R}^d$.
Choose $x_1, \dots, x_i \in \mathbb{R}^d$ such that
$$
\p[(W^N_{t^N_1}, \dots, W^N_{t^N_i}) = (x_1, \dots, x_i)] > 0
$$
and denote by $(Y^{N,x}_t, Z^{N,x}_t, M^{N,x}_t)_{t \ge t^N_i}$ the solution
$(Y^N_t,Z^N_t,M^N_t)$ conditioned on
$$
(W^N_{t^N_1}, \dots, W^N_{t^N_i}) = (x_1, \dots , x_i).
$$
It is adapted to the
filtration $(\tilde{\cal F}_t)_{t \ge t^n_i}$ generated by the $d$-dimensional Brownian motion
$$
\tilde{W}^N_t := W^N_t - W^N_{t^N_i}, \quad t \ge t^N_i,
$$
and solves the BS$\Delta$E
\bea
\label{x} Y^{N,x}_{t^N_j}&=&
Y^{N,x}_{t^N_{j+1}}+ f^N(t^N_{j+1},x_1,\dots, x_{i-1},x_i + \tilde{W}^N,Y^{N,x}_{t^N_j},
Z^{N,x}_{t^N_{j+1}}) \Delta  t^N_{j+1} \nonumber\\
&& - Z^{N,x}_{t^N_{j+1}}\Delta \tilde{W}^N_{t^N_{j+1}}-(M^{N,x}_{t^N_{j+1}}-M^{N,x}_{t^N_j})\\
Y^{N,x}_T &=& \xi^{N,x}
\label{y},
\eea
where
$$
\xi^{N,x} = \varphi^N(x_1, \dots, x_{i-1},x_i + \tilde{W}^N).
$$
Now let $x'_i \in \mathbb{R}^d$ such that
$$
\p[(W^N_{t^N_1}, \dots, W^N_{t^N_i}) = (x_1, \dots, x_{i-1}, x'_i)] > 0
$$
and denote $x' = (x_1, \dots ,x_{i-1}, x'_i)$. If $i = i_N$, one obtains
directly from \eqref{phinlip} that
$$
|Y^{N,x'}_T - Y^{N,x}_T|
= |\varphi^N(x_1, \dots, x'_{i_N}) - \varphi^N(x_1,\dots,x_{i_N})
\le L|x'_{i_N} - x_{i_N}|.
$$
If $i \le i_{N-1}$, note that $|\max  [a_1,a_2] -\max [b_1,b_2]| \le \max [|a_1-b_1|, |a_2-b_2|]$ for
all $a_1,a_2,b_1,b_2 \in \mathbb{R}$. Therefore, one obtains from
Lemma \ref{lemmadualY} for all $j \ge i$,
\beas
&& |Y^{N,x'}_{t^N_j}-Y^{N,x}_{t^N_j}|\\
&=& \Bigg| \max_{\mu \in \crl{\mu',\mu^*}}
\mathbb{E}^{\mu} \edg{\xi^{N,x'}-
\sum_{l=j+1}^{i_N} g^N(t^N_l,x_1, \dots, x_{i-1},x'_i+\tilde{W}^N,Y^{N,x'}_{t^N_{l-1}},\mu_{t^N_l})\Delta t^N_l
\bigg| \tilde{\cal F}^N_{t^N_j}}\\
&& -  \max_{\mu \in \crl{\mu',\mu^*}} \mathbb{E}^{\mu} \edg{\xi^{N,x}-
\sum_{l=j+1}^{i_N} g^N(t^N_l,x_1,\dots,x_{i-1},x_i+\tilde{W}^N,Y^{N,x}_{t^N_{l-1}},\mu_{t^N_l}) \Delta t^N_l
\bigg| \tilde{\cal F}^N_{t^N_j}}\Bigg|\\
&\le& \max_{\mu \in\{\mu',\mu^*\}}
\mathbb{E}^{\mu} \Bigg[|\xi^{N,x'} - \xi^{N,x}| +
\sum_{l=j+1}^{i_N} \big|g^N(t^N_l,x_1, \dots, x_{i-1},x'_i+ \tilde{W}^N, Y^{N,x'}_{t^N_{l-1}},\mu_{t^N_l})\\
&& - g^N(t^N_l,x_1,\dots, x_{i-1},x_i+\tilde{W}^N,Y^{N,x}_{t^N_{l-1}}, \mu_{t^N_l}) \big| \Delta  t^N_l \Big|
\tilde{F}^N_{t^N_j} \Bigg]\\
&\le& (L+KT) |x'_i-x_i| + K \sum_{l=j+1}^{i_N} \N{Y^{N,x'}_{t^N_{l-1}}-Y^{N,x}_{t^N_{l-1}}}_\infty \Delta  t^N_l.
\eeas
It follows from Lemma \ref{lemmaGronwall} that
\be \label{LipY}
\N{Y^{N,x'}_{t^N_i}-Y^{N,x}_{t^N_i}}_\infty \le 2(L+KT)\exp(KT) |x'-x|.
\ee
To see that this implies $|Z^N_{t^N_{i}}| \le 2 \sqrt{d}(L+KT) \exp(KT)$, note that
because $Y^N_{t^N_i}$ is $(\F^N_{t^N_t})$-measurable, there exist functions
$y^N_{t^N_i}:\mathbb{R}^{i \times d}\to \mathbb{R}$ such that
$$
Y^N_{t^N_{i}} = y^N_{t^N_{i}}(W^N_{t^N_1},\dots,W^N_{t^N_i}).
$$
So the components of $Z^N_{t^N_i}$ satisfy
\beas
&& |Z^{N,k}_{t^N_i}| = \frac{\Big|\E{Y^N_{t^N_i}\Delta W^{N,k}_{t^N_{i}}}
\big| {\cal F}^N_{t^N_{i-1}} \Big|}{\Delta  t^N_i}\\
&=& \frac{\bigg|\E{\Big(y^N_{t^N_{i}}(W^N_{t^N_{1}},\dots,W^N_{t^N_{i-1}},W^N_{t^N_i})
-y^N_{t^N_i}(W^N_{t^N_{1}},\dots,W^N_{t^N_{i-1}},W^N_{t^N_{i-1}}) \Big) \Delta W^{N,k}_{t^N_i}
\big| {\cal F}^N_{t^N_{i-1}}} \bigg|}{\Delta  t^N_i}\\
&\le& \frac{2(L + KT) \exp(KT)}{\Delta t^N_i}
\E{\big|\Delta W^{N,k}_{t^N_i} \big|^2}
= 2(L + KT) \exp(KT),
\eeas
which entails $|Z^N_{t^N_i}| \le 2 \sqrt{d}(L + KT) \exp(KT)$.
\end{proof}

\begin{lemma} \label{lemmaYbounded}
Assume $(Y,Z)$ is a solution of the BSDE \eqref{bsde} corresponding to a
bounded terminal condition such that $Z$ is bounded and $f$ satisfies
{\rm (f2)--(f4)}. Then $Y$ is bounded.
\end{lemma}

\begin{proof}
Since $Z$ is bounded, one can assume without loss
of generality that the driver $f$ is Lipschitz in $y$ and $z$ with
Lipschitz-constant $b \in \mathbb{R}_+$. By condition (f2), there exists a constant
$a \in \mathbb{R}_+$ such that $f(t,w,0,0) \le a$ for all $t$ and $w$. Therefore,
$$
f(t,w,y,z) \le f'(t,w,y,z) := a + b(|y| + |z|).
$$
Since $f'$ is Lipschitz in $y$ and $z$, it follows from Pardoux and Peng (1990) that
the BSDE with driver $f'$ and terminal condition $\hat{\xi} := \N{\xi}_{\infty}$ has
a unique solution $(\hat{Y},\hat{Z})$, which is easily verified to be
$$
\hat{Y}_t = \brak{\hat{\xi} + \frac{a}{b}} e^{b(T-t)} - \frac{a}{b}, \quad
\hat{Z}_t = 0,
$$
and it follows from the comparison result shown in
El Karoui et al. (1997) that $Y_t \le \hat{Y}_t$ for all $t$.
Similarly, one obtains that $Y$ is bounded from below.
\end{proof}

\noindent
{\bf Proof of Theorem \ref{thm1}.}\\
By assumption, there exists a constant $L \in \mathbb{R}_+$ such that
$|\varphi(w_1)-\varphi(w_2)| \le L \N{w_1-w_2}_{\infty}$ for all $w_1,w_2 \in C^d[0,T]$.
It follows from Proposition \ref{propex} and Lemma \ref{lemmaZbound}
that there exists an $N_0 \in \mathbb{N}$ such that for every $N \ge N_0$, the $N$-th BS$\Delta$E has
a unique solution $(Y^N,Z^N,M^N)$ and
$$
\sup_{0 \le t \le T} |Z^N_t| \le 2 \sqrt{d}(L + KT) \exp(KT).
$$
One can choose a function
$\tilde{f} : [0,T] \times C^d[0,T] \times \mathbb{R} \times \mathbb{R}^d \to \mathbb{R}$
that agrees with $f$ for $|z| \le 2 \sqrt{d}(L + KT) \exp(KT)$,
satisfies (f1)--(f3) and is Lipschitz-continuous in $z$.
From Pardoux and Peng (1990) one obtains that the BSDE \eqref{bsde} with driver
$\tilde{f}$ has a unique solution $(Y,Z)$, and it is a consequence of Theorem 12 of Briand et al. (2002) that
$$
\sup_t\Big(|Y^{N}_t- Y_t|+ |\int_0^t Z^N_s dW^N_s-\int_0^t Z_s dW_s|+|M^{N}_t|\Big)
\to 0 \quad \mbox{in } L^2,
$$
and
\be
\label{convz4}
\sup_t \brak{\sum_{k=1}^d \abs{\int_0^t Z^{N,k}_s d \ang{W^N}_s - \int_0^t Z^k_s ds}^2
+ \abs{\int_0^t |Z^N_s|^2 d \ang{W^N}_s - \int_0^t |Z_s|^2 ds}}
\to 0 \quad \mbox{in } L^1.
\ee
(Briand et al. (2002) prove this result for the case where
the Brownian motion $W$ is one-dimensional and drivers are RCLL.
But we show in the Appendix that it also holds in our setup.)
It follows from \eqref{convz4} that
$$
|Z_t| \le 2 \sqrt{d}(L + KT) \exp(KT) \quad dt \times d\p \mbox{-almost everywhere.}
$$
So $(Y,Z)$ is also a solution of the BSDE \eqref{bsde} with driver $f$.

If one replaces $f$ by a driver $f' \ge f$ satisfying (f1)--(f4)
and $\xi$ by a terminal condition $\xi' \ge \xi$ of the form $\xi' = \varphi'(W)$
for a Lipschitz-continuous function $\varphi' : C^d[0,T] \to \mathbb{R}$, the
BSDE \eqref{bsde} has a solution $(Y',Z')$ such that $Z'$
is bounded by a constant $C \in \mathbb{R}_+$. So one can modify $f$ and $f'$ for
$$
|z| > C \vee 2 \sqrt{d}(L + KT) \exp(KT)
$$
such that they satisfy (f1)--(f3) and are Lipschitz-continuous in $z$.
But then it follows from the comparison result proved in El Karoui et al. (1997) that
$Y'_t \ge Y_t$ for all $t$. In particular, $(Y,Z)$ is the only solution of \eqref{bsde}
such that $Z$ is bounded. Finally, if $\varphi$ is bounded, one obtains from Lemma
\ref{lemmaYbounded} that $Y$ is bounded as well.
\unskip\nobreak\hfill$\Box$\par\addvspace{\medskipamount}

\setcounter{equation}{0}
\section{Convex duality and comparison}

As in the discrete-time case we exploit the convexity of $f$ to derive
convex dual representations for solutions of BSDEs (see Lemma \ref{lemmaYcondexp} below).
If $f$ does not depend on $y$, the representation is explicit and  coincides with the
ones in Barrieu and El Karoui (2009) or Delbaen et al. (2009). But if $f$ depends on $y$, it is
implicit as in the discrete-time case.

Denote the set of all $d$-dimensional BMO processes $\mu$ by BMO.
The norm $\N{\mu}_{\rm BMO}$ is the smallest number $c$ such that
$$
\sqrt{\E{\int_{\tau}^T |\mu_s|^2 ds| \F_\tau}} \le c
$$
for all stopping times $\tau$ taking values in $[0,T]$. It is well-known
from Kazamaki (1994) that for every $\mu \in {\rm BMO}$,
$$
\Gamma^{\mu}_t = \exp \brak{\int_0^t \mu_s dW_s - \frac{1}{2} \int_0^t |\mu_s|^2 ds},
\quad 0 \le t \le T,
$$
is a martingale. By Girsanov's theorem, $\p^{\mu} = \Gamma^{\mu}_T \cdot \p$
defines a probability measure equivalent to $\p$
under which $W^{\mu}_t = W_t - \int_0^t \mu_s ds$ is a $d$-dimensional Brownian motion.
Moreover, every BMO process with respect
to $\p$ is also a BMO process with respect to $\p^{\mu}$.

Before we can turn to convex dual representations, we need the following
technical

\begin{lemma} \label{lemmasemi}
Let $Y^n$, $n \in \mathbb{N}$, be a sequence of $({\cal F}_t)$-semimartingales
with canonical decompositions
$$
Y^n_t = Y^n_0 + U^n_t + V^n_t.
$$
Assume the $Y^n$ are uniformly bounded by a constant $C \in \mathbb{R}_+$ and
there exists $b \in \mathbb{R}_+$ such that for all
$n \in \mathbb{N}$, $V^n_t + bt$ is increasing. Then there exist
{\rm BMO} processes $Z^n$ such that $U^n_t = \int_0^t Z^n_s dW_s$ and
\be \label{Aest}
\mathbb{E} \edg{\int_{\tau}^T |Z^n_s|^2ds \mid {\cal F}_{\tau}}
+ \E{\int_0^T |dV^n_s|} \le 4 e^{2C+2|b|T}+|b|T
\ee
for all stopping times $\tau$ and $n \in \mathbb{N}$.
In particular, $\sup_n \N{Z^n}_{\rm BMO} < \infty$.
\end{lemma}

\begin{proof}
The canonical decomposition of the semimartingale $\tilde{Y}^n_t = Y^n_t + bt$ is
$$
\tilde{Y}^n_t = Y^n_0 + U^n_t + \tilde{V}^n_t,
$$
for the increasing finite variation process $\tilde{V}^n_t = V^n_t + bt$.
Since $(W_t)$ has the predictable representation property, there
exist $\mathbb{R}^d$-valued $({\cal F}_t)$-predictable processes $Z^n$ such that
$U^n_t = \int_0^t Z^n_s dW_s$. In particular, $U^n_t$ is continuous. Hence,
$\Delta \tilde{Y}^n_t = \Delta \tilde{V}^n_t \ge 0$ for all $t$. For fixed
$n \in \mathbb{N}$, let $\sigma_m$, $m \in \mathbb{N}$, be an increasing sequence
of $[0,T]$-valued stopping times such that $\p[\sigma_m = T] \uparrow 1$ and
$U^n_{t \wedge \sigma_m}$ is a martingale for every $m$. It follows from
It\^o's formula that for every $[0,T]$-valued stopping time $\tau$,
\beas
\exp(\tilde{Y}^n_{\sigma_m})
&=& \exp(\tilde{Y}^n_{\tau \wedge\sigma_m}) + \int_{\tau}^{T} \exp(\tilde{Y}^n_s) dU^n_{s \wedge \sigma_m}
+ \frac{1}{2} \int_{\tau}^{T} \exp(\tilde{Y}^n_s)d \ang{U^n}_{s \wedge \sigma_m}\\
&+& \int_{\tau+}^{T} \exp(\tilde{Y}^n_{s-}) d\tilde{V}^n_{s \wedge \sigma_m}
+ \sum_{\tau < s \le \sigma_m} \Delta \exp(\tilde{Y}^n_s) - \exp(\tilde{Y}^n_{s-}) \Delta \tilde{Y}^n_s.
\eeas
Since
$\sum_{\tau < s \le \sigma_m} \Delta \exp(\tilde{Y}^n_s) - \exp(\tilde{Y}^n_{s-}) \Delta \tilde{Y}^n_s \ge 0$,
one can take conditional expectation to obtain
$$
\mathbb{E}_{{\cal F}_{\tau \wedge \sigma_m}} \edg{
\exp(\tilde{Y}^n_{\sigma_m})} \ge
\mathbb{E}_{{\cal F}_{\tau \wedge \sigma_m}} \edg{
\int_{\tau}^T \frac{1}{2} \exp(\tilde{Y}^n_s) d \ang{U^n}_{s \wedge \sigma_m}
+ \int_{\tau+}^{T} \exp(\tilde{Y}^n_{s-}) d\tilde{V}^n_{s \wedge \sigma_m}}.
$$
But since $\ang{U^n}$ and $\tilde{V}^n$ are increasing and $\tilde{Y}^n$
is bounded by $\tilde{C} = C + |b|T$, one obtains
$$
\exp(\tilde{C}) \ge
\frac{1}{2} \exp(-\tilde{C}) \mathbb{E}_{{\cal F}_{\tau \wedge \sigma_m}} \edg{
\ang{U^n}_{\sigma_m} -\ang{U^n}_{\tau \wedge \sigma_m}
+ \tilde{V}^n_{\sigma_m} -\tilde{V}^n_{\tau\wedge \sigma_m}},
$$
and therefore,
\be \label{mittau}
\mathbb{E}_{{\cal F}_{\tau \wedge \sigma_m}} \edg{
\ang{U^n}_{\sigma_m}- \ang{U^n}_{\tau \wedge \sigma_m}
+ \tilde{V}^n_{\sigma_m}- \tilde{V}^n_{\tau\wedge \sigma_m}} \le 2e^{2\tilde{C}}.
\ee
By choosing $\tau =0$ and letting $m$ converge to infinity, one obtains from Beppo Levi's
monotone convergence theorem that
$$
\E{\ang{U^n}_T} \le 2e^{2 \tilde{C}},
$$
which, by the Burkholder--Davis--Gundy inequality, implies that $U$ is a square-integrable martingale.
So one may choose $\sigma_m = T$, and it follows from
\eqref{mittau} that
\be \label{MandA}
\mathbb{E}_{{\cal F}_{\tau}} \edg{
\ang{U^n}_T - \ang{U^n}_{\tau} + \tilde{V}^n_T - \tilde{V}^n_{\tau}} \le 2e^{2\tilde{C}}.
\ee
Using $\ang{U^n}_T - \ang{U^n}_{\tau} = \int_{\tau}^T |Z^n_s|^2ds$
and the fact that $\tilde{V}$ is increasing, one obtains
$$
\mathbb{E}_{{\cal F}_{\tau}} \edg{\int_{\tau}^T |Z^n_s|^2ds} +
\E{\int_0^T |d \tilde{V}^n_s|} \le 4 e^{2 \tilde{C}},
$$
which implies \eqref{Aest}.
\end{proof}

\begin{Remark} \label{Adecreasing}
By replacing $Y$ with $-Y$, one sees that Lemma \ref{lemmasemi} also holds if there
exist constants $C$ and $b$ such that for every $n \in \mathbb{N}$,
$Y^n$ is bounded by $C$ and the process $A^n_t + bt$ is decreasing.
\end{Remark}

Let us denote by $g$ the convex conjugate of $f$ with respect to $z$, that is,
$$
g(t,w,y,\mu) = \sup_z \crl{z \mu - f(t,w,y,z)}, \quad \mu \in \mathbb{R}^d.
$$
$g$ maps $[0,T] \times C^d[0,T] \times \mathbb{R} \times \mathbb{R}^d$ to
$\mathbb{R} \cup \crl{\infty}$ and inherits condition (f3) from $f$, that is,
$$
|g(t,w_{1},y_1,\mu) - g(t,w_{2},y_2,\mu)| \le K \big(\sup_{0 \le s \le t}
|w_1(s) -w_2(s)|+|y_1-y_2|\big)
$$
for all $t, w_1, w_2, y_1, y_2, \mu$.

\begin{lemma} \label{lemmaYcondexp}
Suppose that $(Y,Z,A)$ is a supersolution of the BSDE
\eqref{bsde} such that $Y$ is bounded. If $f$ satisfies {\rm (f5)} or $Z$ is BMO, then
\be \label{dualinequ}
Y_{\sigma} \ge \mathbb{E}^{\mu}_{{\cal F}_{\sigma}} \edg{Y_\tau-\int_\sigma^\tau g(s,W,Y_s,\mu_s)ds}
\ee
for every $\mu \in {\rm BMO}$ and all stopping times
$0 \le \sigma \le \tau \le T$. If $(Y,Z)$ is a solution of the BSDE such that $Z$ is bounded
and $f$ satisfies {\rm (f1), (f4)} and {\rm (f5)},
there exists a bounded $\mathbb{R}^d$-valued $({\cal F}_t)$-predictable process $\mu^*$ such that
\be \label{dualequ}
Y_{\sigma} = \mathbb{E}^{\mu^*}_{{\cal F}_{\sigma}} \edg{Y_\tau-\int_\sigma^\tau
g(s,W,Y_s,\mu^*_s)ds}
\ee
for all stopping times $0 \le \sigma \le \tau \le T$.
\end{lemma}

\begin{proof}
If $Y$ is bounded and $f$ satisfies condition (f5), one obtains from Lemma
\ref{lemmasemi} and Remark \ref{Adecreasing} applied to $Y^n_t
= Y_t$, $U^n_t = \int_0^t Z_s dW_s$ and $V^n_t = - \int_0^t f(s,W,Y_s,Z_s)ds - A_t$
that $Z$ is BMO. But then it is also BMO with respect to $\p^{\mu}$ for every
$\mu \in {\rm BMO}$; see Section 3.3 of Kazamaki (1994).
It follows that
\bea \label{esssup}
Y_{\sigma} &=& \mathbb{E}^{\mu}_{{\cal F}_{\sigma}} \edg{Y_\tau+\int_\sigma^\tau
f(s,W,Y_s,{Z}_{s})ds - \int_\sigma^\tau {Z}_s dW_s +A_\tau-A_\sigma}\nonumber\\
&\ge& \mathbb{E}^{\mu}_{{\cal F}_{\sigma}} \edg{Y_\tau-\int_\sigma^\tau
\big[\mu_s {Z}_s-f(s,W,Y_s,Z_s)\big] ds
- \int_\sigma^\tau Z_s (dW_s- \mu_s ds)}\label{gemu2} \\
&=& \mathbb{E}^{\mu}_{{\cal F}_{\sigma}} \edg{Y_{\tau}-\int_\sigma^\tau
\big[\mu_s {Z}_s-f(s,W,Y_s,{Z}_{s})\big]ds}\nonumber\\
\label{gemu} &\ge& \mathbb{E}^{\mu}_{{\cal F}_{\sigma}}
\edg{Y_{\tau}-\int_\sigma^\tau g(s,W,Y_s,\mu_{s})ds}.
\eea
Of course, \eqref{gemu2} becomes an equality
if $Y$ is not only a supersolution but a true solution.
Furthermore, if $f$ satisfies (f1), it follows from Lemma 6.2 in Cheridito and
Stadje (2009) that there exists an $\mathbb{R}^d$-valued
$({\cal F}_t)$-predictable process $(\mu^*_t)$ such that
$\mu^*_t$ is in the subgradient $\partial f(t,W,Y_t,Z_t)$ of
$f$ with respect to $z$ for $dt \times d \p$-almost all
$(t,\omega)$. If $Z_t$ is bounded, it follows from (f4) that
$\mu^*$ is bounded too. So $\mathbb{P}^{\mu^*}$ is a
well-defined probability measure, and inequality \eqref{gemu}
becomes an equality for $\mu = \mu^*$.
\end{proof}

\begin{definition}
We say a supersolution $(Y,Z,A)$ of the BSDE \eqref{bsde}
satisfies assumption {\rm (A)} if for every constant
$\varepsilon > 0$, there exists a $\mu \in {\rm BMO}$ such that
\be \label{epscond}
Y_t \le \mathbb{E}^{\mu}_{{\cal F}_t} \edg{\xi-\int_t^T
g(s,W,Y_s,\mu_s)ds} + \varepsilon \quad \mbox{for all } 0 \le t \le T.
\ee
\end{definition}

Note that if $(Y,Z,A)$ is a supersolution of the BSDE
\eqref{bsde} satisfying assumption (A) such that $Y$ is
bounded, then \beas Y_t \le \esssup_{\mu \in {\rm BMO}}
\mathbb{E}^{\mu}_{{\cal F}_t} \edg{ \xi - \int_t^T g(s,W,Y_s,\mu_s)ds}
\quad \mbox{for all } 0 \le t \le T.
\eeas

The following proposition gives a comparison result:

\begin{proposition} \label{increasing}
Assume $f$ is increasing in $y$ and $(Y,Z,A)$ is a supersolution of the BSDE \eqref{bsde}
such that $Y$ is bounded and fulfils assumption {\rm (A)}. Then
if $(Y',Z',A')$ is a supersolution of \eqref{bsde} with bounded
terminal condition $\xi' \ge \xi$ and driver $f' \ge f$ satisfying {\rm (f5)}
such that $Y'$ is bounded, one has $Y'_t \ge Y_t$ for all $0 \le t \le T$.
\end{proposition}

\begin{proof}
Fix $\varepsilon>0$. There exists a BMO
process $\mu$ such that for all $t \in [0,T]$,
$$
Y_t \le \mathbb{E}^{\mu}_{{\cal F}_t} \edg{\xi-\int_t^T g(s,W,Y_s,\mu_s)ds} +\varepsilon
\quad \mbox{for all } 0 \le t \le T.
$$
Define
$$
g'(t,w,y,\mu)= \sup_{z \in \mathbb{R}^d} \crl{\mu z-f'(t,w,y,z)}, \quad \mu \in \mathbb{R}^d.$$
Since $f' \ge f$, one has $g' \le g$, and therefore,
$$
Y'_t - \mathbb{E}^{\mu}_{{\cal F}_t} \edg{\xi'-\int_t^T g(s,W,Y'_s,\mu_s) ds}
\ge Y'_t- \mathbb{E}^{\mu}_{{\cal F}_t} \edg{\xi'-\int_t^T g'(s,W,Y'_s,\mu_s)ds}
\ge 0,
$$
where the last inequality follows from Lemma \ref{lemmaYcondexp}.
Since $f$ is increasing in $y$, $g$ is decreasing in $y$. So
$g(t,w,y_1,z) - g(t,w,y_2,z) \le 0$ for all $y_1 \ge y_2$. On the other hand,
if $y_1 \le y_2$, one has
$$
0 \le g(t,w,y_1,z)-g(t,w,y_2,z) \le K (y_2 - y_2).
$$
Hence,
\be
\label{pluslipschitz}
\big(g(t,w,y_1,z)-g(t,w,y_2,z)\big)^+\leq K(y_2-y_1)^+ \quad \mbox{for all } y_1,y_2 \in \mathbb{R}.
\ee
It follows that
\beas
&& (Y_t - Y'_t)^+\\
&\le& \brak{\varepsilon +
\mathbb{E}^{\mu}_{{\cal F}_t} \edg{\xi - \int_t^T g(s,W,Y_s,\mu_s)ds}
- \mathbb{E}^{\mu}_{{\cal F}_t} \edg{\xi' -\int_t^T g(s,W,Y'_s,\mu_s)ds}}^+\\
&\le& \varepsilon + \mathbb{E}^{\mu}_{{\cal F}_t} \edg{\int_t^T \brak{g(s,W,Y'_s,\mu_s)
- g(s,W,Y_s,\mu_s)}^+ ds}\\
&\le& \varepsilon + \mathbb{E}^{\mu}_{{\cal F}_t} \edg{\int_t^T K (Y_s - Y'_s)^+ ds}.
\eeas
In particular,
$$
\mathbb{E}^{\mu} \edg{(Y_t - Y'_t)^+} \le \varepsilon +
K \int_t^T \mathbb{E}^{\mu} \edg{(Y_s - Y'_s)^+} ds \quad
\mbox{for all } t,
$$
and one obtains from Gronwall's Lemma that
$$
\mathbb{E}^{\mu} \edg{(Y_t - Y'_t)^+} \le \varepsilon \exp\{K(T-t)\}, \quad 0 \le t \le T.
$$
Since $\varepsilon >0$ can be chosen arbitrarily, one gets $Y_t \le Y'_t$ for all $t$.
\end{proof}

The following proposition gives a comparison result for the case when
$f$ is decreasing in $y$:

\begin{proposition} \label{decreasing}
Assume $f$ is decreasing in $y$ and $(Y,Z,A)$ is a supersolution of the
BSDE \eqref{bsde} such that $Y$ is bounded and satisfies assumption
{\rm (A)}. If $(Y',Z',A')$ is a
supersolution of \eqref{bsde} with bounded terminal condition
$\xi' \ge \xi$ and driver $f' \ge f$ satisfying {\rm (f5)} such that
$Y'$ is bounded, then $Y'_t \ge Y_t$ for all $0 \le t \le T$.
\end{proposition}

\begin{proof}
We prove this proposition by contradiction. Set
$$
g'(t,w,y,\mu)=\sup_z \crl{\mu z-f'(t,w,y,z)}, \quad \mu \in \mathbb{R}^d.
$$
Since $f' \ge f$, one has $g' \le g$.
Assume that there exists $t \in[0,T]$ such that
$\p[Y'_t < Y_t] > 0$ and define $\tau := \inf\{s > t : Y'_s \ge Y_s\}$.
Since $Y'_T = \xi' \ge \xi =Y_T$, one has $t \le \tau \le T$.
By conditioning on $\crl{Y'_t < Y_t}$, one can assume that
$\p[Y'_t < Y_t] = 1$. Then
\beas
&& \esssup_{\mu\in {\rm BMO}} \mathbb{E}^{\mu}_{{\cal F}_t} \edg{Y'_{\tau}
- \int_t^{\tau} g(s,W,Y'_s,\mu_s)ds}\\
&\le& \esssup_{\mu\in {\rm BMO}} \mathbb{E}^{\mu}_{{\cal F}_t} \edg{Y'_{\tau}
- \int_t^{\tau} g'(s,W,Y'_s,\mu_s)ds} \le Y'_t\\
&<& Y_t \le \esssup_{\mu\in {\rm BMO}} \mathbb{E}^{\mu}_{{\cal F}_t} \edg{Y_{\tau}
- \int_t^{\tau} g(s,W,Y_s,\mu_s)ds}.
\eeas
However, since $f$ is decreasing in $y$, $g$ is increasing in $y$.
By the definition of $\tau$, one has $Y'_s \le Y_s$ for $t \le s < \tau$ and hence,
$$
\int_t^\tau g(s,W,Y'_s,\mu_s)ds \le \int_t^\tau g(s,W,Y_s,\mu_s)ds.
$$
On the other hand, $Y_{\tau} \le Y'_{\tau}$, and therefore,
$$
\mathbb{E}^{\mu}_{{\cal F}_t} \edg{Y'_\tau-\int_{t}^\tau g(s,W,Y'_s,\mu_s)ds \mid {\cal F}_t} \ge
\mathbb{E}^{\mu}_{{\cal F}_t} \edg{Y_\tau-\int_{t}^\tau g(s,W,Y_s,\mu_s)ds \mid {\cal F}_t}
$$
for all $\mu \in {\rm BMO}$, a contradiction.
\end{proof}

\setcounter{equation}{0}
\section{Proofs of Theorems \ref{thm2} and \ref{thm3}}

For $p \in [1,\infty]$, denote by $\mathcal{S}^p$ the space of all
$(\F_t)$-semimartingales $X$ such that
$$
||X||_{\mathcal{S}^p} := \N{\sup_{0 \le t \le T} |X_t|}_{L^p} < \infty.
$$
and by ${\cal H}^p$ the space of all special $(\F_t)$-semimartingales $X$ with
canonical decomposition $X = X_0 + U + V$ satisfying
$$
||X||_{\mathcal{H}^p} := \N{X_0}_{L^p} + \N{\edg{U,U}_T^{1/2}}_{L^p} +
\N{\int_0^T |dV_s|}_{L^p} < \infty.
$$

\begin{lemma} \label{lemmalipschitz}
Let $(Y^n,Z^n)$, $n = 1,2$, be solutions of the BSDE \eqref{bsde}
corresponding to bounded terminal conditions $\xi^n$ such that $Z^n$ are bounded and
$f$ satisfies {\rm (f1)--(f5)}. Then $Y^1$ and $Y^2$ are bounded and
$$
\N{Y^1 - Y^2}_{{\cal S}^{\infty}} \le \exp\{KT\}||\xi^1-\xi^2||_{L^{\infty}}.
$$
\end{lemma}

\begin{proof}
By Lemma \ref{lemmaYbounded}, $Y^1$ and $Y^2$ are
bounded. So it follows from Lemma \ref{lemmaYcondexp} that
there exist bounded $\mathbb{R}^d$-valued $({\cal F}_t)$-predictable processes
$\mu^n$, $n=1,2$, such that
\beas
Y^n_t &=& \esssup_{\mu \in {\rm BMO}} \mathbb{E}^{\mu}_{{\cal F}_t}
\edg{\xi^n - \int_t^T g(s,W,Y^n_s,\mu_s)ds}\\
&=& \mathbb{E}^{\mu^n}_{{\cal F}_t} \edg{\xi^n - \int_t^T g(s,W,Y^n_s,\mu^n_s)ds},
\eeas
and one obtains as in the proof of Lemma \ref{lemmaZbound} that
\beas
|Y^1_t-Y^2_t| &\le& \sup_{\mu \in \{\mu^1,\mu^2\}} \mathbb{E}^{\mu}_{{\cal F}_t}
\edg{|\xi^1-\xi^2|+ \int_t^T |g(s,W,Y^1_s,\mu_s)- g(s,W,Y^2_s, \mu_s)|ds}\\
&\le& ||\xi^1 - \xi^2||_\infty + K \int_t^T ||Y^1_s -Y^2_s||_\infty ds.
\eeas
Now the lemma follows from Gronwall's lemma.
\end{proof}

We need the following result of Barlow and Protter (1990):

\begin{theorem} \label{thmBP} {\rm (Barlow and Protter, 1990)}\\
Let $(Y^n_t)_{0 \le t \le T}$, $n \in \mathbb{N}$, be a
sequence of semimartingales in $\mathcal{H}^1$ over
a filtered probability space with canonical decompositions
$Y^n= Y^n + U^n + V^n$ such that
\be \label{1a}
\sup_n \N{U^n}_{{\cal S}^1} \le K \quad \mbox{and} \quad
\sup_n \N{V^n}_{{\cal H}^1} \le  K \quad \mbox{for some } K \in \mathbb{R}_+
\ee
and $Y$ a RCLL process on the same probability space such that
$$
\lim_{n \to \infty} \N{Y^n - Y}_{{\cal S}^1} = 0.
$$
Then $Y$ is a semimartingale in $\mathcal{H}^1$ with canonical decomposition
$Y = Y_0 + U + V$ satisfying
$$
\N{U}_{{\cal S}^1} \le K, \quad \N{V}_{{\cal H}^1} \le K
$$
and
$$
\lim_{n \to \infty} \N{U^n - U}_{\mathcal{H}^1} = 0 \quad \mbox{and} \quad
\lim_{n \to \infty} \N{V^n - V}_{{\cal S}^1} = 0.
$$
\end{theorem}

Now we are ready for the proof of Theorem \ref{thm2}:\\
{\bf Proof of Theorem \ref{thm2}.}
It follows from Theorem \ref{thm1} that $Y^n$ is bounded for all $n$ and from
Lemma \ref{lemmalipschitz} that
$$
\N{Y^m-Y^n}_{S^{\infty}} \le \exp\{KT\}||\xi^m-\xi^n||_{L^{\infty}}.
$$
Hence, $Y^n$ is a Cauchy sequence in $\mathcal{S}^\infty$. So there exists a continuous process
$Y \in {\cal S}^{\infty}$ such that $\N{Y^n-Y}_{{\cal S}^{\infty}} \to 0$ for $n \to \infty$.
It follows that $Y_T = \xi$.
To see that $Y$ is a supersolution of the BSDE \eqref{bsde}, note that
$Y^n$ is a continuous semimartingale with canonical decomposition
$Y^n = Y^n_0 + U^n + V^n$, where
$$
U^n_t = \int_0^t Z^n_s dW_s \quad \mbox{and} \quad
V^n_t = - \int_0^t f(s,W,Y_{n,s},Z^n_s) ds.
$$
Due to (f5) and the fact that the $Y^n$ are uniformly bounded it follows from
Lemma \ref{lemmasemi} and Remark \ref{Adecreasing} that
there exists a constant $C$ such that
$$
\E{\int_{\tau}^T |Z^n_s|^2 ds \mid {\cal F}_{\tau}} + \E{\int_0^T |f(s,W,Y^n_s,Z^n_s)| ds} \le C
$$
for all $n$ and every stopping time $\tau$.
In particular, $\sup_n \N{Z^n}_{\rm BMO} < \infty$ and $\sup_n \N{V^n}_{{\cal H}^1} < \infty$.
It follows that $\sup_n \N{U^n}_{{\cal H}^2} < \infty$, which implies that
$Y^n \in {\cal H}^1$ and $\sup_n \N{U^n}_{{\cal S}^1} < \infty$.
So the assumptions of Theorem \ref{thmBP} are satisfied,
and it follows that $Y$ is a semimartingale in ${\cal H}^1$
with canonical decomposition $Y_t = Y_0 + U_t + V_t$ such that
$U^n \to U$ in ${\cal H}^1$ and $V^n \to V$ in ${\cal S}^1$.
By the predictable representation property of $(W_t)$, there exists a $d$-dimensional
$({\cal F}_t)$-predictable process $Z$ such that $U_t = \int_0^t Z_s dW_s$ and
$$
\E{\sqrt{\int_0^T |Z^n_s-Z_s|^2 ds}} \to 0.
$$
By passing to a subsequence, one can assume that
\be \label{l2conv}
\int_0^T |Z^n_s- Z_s|^2 ds \to 0 \quad \mbox{almost surely.}
\ee
For every stopping time $\tau$ and $B \in {\cal F}_{\tau}$, one obtains from
Fatou's lemma that
$$
\E{1_B \int_{\tau}^T |Z_s|^2 ds} \le \liminf_n \E{1_B \int_{\tau}^T |Z^n_s|^2 ds},
$$
which shows that $Z$ belong to ${\rm BMO}$. It follows from
\eqref{l2conv} that for almost all $\omega$, one can pass to another
subsequence such that $Z^n_s(\omega) \to Z_s(\omega)$ for Lebesgue-almost all
$s \in [0,T]$. Hence, due to condition (f5), one can deduce from Fatou's lemma that
\beas
&& -V_t(\omega) + V_r(\omega) = \lim_n - V^n_t(\omega) + V^n_r(\omega)\\
&=& \lim_n \int_r^t f(s,W(\omega),Y^n_s(\omega),Z^n_s(\omega))ds
\ge \int_r^t f(s,W(\omega),Y_s(\omega),Z_s(\omega))ds
\eeas
for all $r < t$. So $A_t = - V_t - \int_0^t f(s,W,Y_s,Z_s)ds$ is a continuous increasing process
starting at $0$ such that
$$
Y_t = \xi + \int_t^T f(s,W,Y_s,Z_s) ds - \int_t^T Z_s dW_s + A_T - A_t.
$$
This shows that $(Y,Z,A)$ is a supersolution of the BSDE \eqref{bsde}
such that $Y$ is bounded and continuous.

Now assume that $f$ is increasing or decreasing in $y$. To see that
then $Y$ satisfies bounded comparison from above, note that one obtains from the
second part of Lemma \ref{lemmaYcondexp} that
$$
Y^n_t = \esssup_{\mu\in {\rm BMO}} \mathbb{E}^{\mu}_{{\cal F}_t} \edg{\xi^n-\int_t^T g(s,W,Y^n_s,\mu_s)ds}
= \mathbb{E}^{\mu^n}_{\F_t} \edg{\xi^n-\int_t^T g(s,W,Y^n_s,\mu^n_s)ds}
$$
for a sequence $\mu^n \in {\rm BMO}$.
We will show that this implies that $Y$ satisfies assumption (A).
For given $\varepsilon > 0$, choose $n \in \mathbb{N}$ so large that
$$
\N{Y^n - Y}_{S^\infty} \le \min \brak{\frac{\varepsilon}{3KT},\frac{\varepsilon}{3}}.
$$
Then
\beas
\frac{\varepsilon}{3} &\ge& \N{Y^n - Y}_{S^{\infty}}\\
&=& \N{\sup_t \Big|\mathbb{E}^{\mu^n}_{\F_t} \edg{\xi^n -\int_t^T
g(s,W,Y^n_s,\mu^n_s)ds} - Y_t\Big|}_{L^{\infty}}\\
&\ge& \N{\sup_t \Big|\mathbb{E}^{\mu^n}_{\F_t} \edg{\xi-\int_t^T g(s,W,Y^n_s,\mu^n_s)ds
} - Y_t\Big|}_{L^{\infty}} - \N{Y^n_T- Y_T}_{L^{\infty}} \\
&\ge&  \N{\sup_t \Big|\mathbb{E}^{\mu^n}_{\F_t} \edg{\xi-\int_t^T g(s,W,Y_s,\mu^n_s)ds
} - Y_t \Big|}_{L^{\infty}} -K \N{\int_0^T |Y^n_s-Y_s| ds }_{L^{\infty}} -\frac{\varepsilon}{3}\\
&\ge& \N{\sup_t \Big|\mathbb{E}^{\mu^n}_{\F_t} \edg{\xi-\int_t^T g(s,W,Y_s,\mu^n_s) ds
} - Y_t \Big|}_{L^{\infty}} -\frac{2 \varepsilon}{3}.
\eeas
In particular,
$$
Y_t \le \mathbb{E}^{\mu^n}_{{\cal F}_t} \edg{\xi-\int_t^T
g(s,W,Y_s,\mu^n_s)ds} + \varepsilon \quad \mbox{for all } t \in [0,T].
$$
This shows that $Y$ satisfies assumption (A). Now if
$(Y',Z')$ is a solution to the BSDE \eqref{bsde} with bounded terminal
condition $\xi' \ge \xi$ and driver $f' \ge f$ such that $Y'$ is bounded, then
$f'$ satisfies condition (f5). So it follows from Proposition \ref{increasing} or
Proposition \ref{decreasing} that $Y'_t \ge Y_t$ for all $t$.
\unskip\nobreak\hfill$\Box$\par\addvspace{\medskipamount}

\noindent
{\bf Proof of Theorem \ref{thm3}.} We know from Theorem
\ref{thm1} that $Y^n$ is an increasing sequence. By Lemma \ref{lemmalipschitz},
it is bounded in ${\cal S}^{\infty}$. So it
converges pointwise to a bounded predictable process $Y$. By Lemma \ref{lemmaYcondexp}, one has
for all $t$, \beas \nonumber Y_t&=& \sup_n Y^n_t
= \sup_n\esssup_{\mu\in\mathcal{B}} \Emustetig{\xi^n-\int_t^T g(s,W,Y^n_s,\mu_s)ds}\nonumber\\
&=& \esssup_{\mu \in\mathcal{B}}\sup_n \Emustetig{\xi^n-\int_t^T g(s,W,Y^n_s,\mu_s)ds}\nonumber\\
&=& \esssup_{\mu\in \mathcal{B}}\Emustetig{\xi-\int_t^T
g(s,W,Y_s,\mu_s)ds}\label{geq3}, \eeas where the last equality
follows from Beppo Levi's monotone converge theorem.
Now one deduces as in Proposition 2.1 and Theorem 2.1 of
Delbaen et al. (2009) that there exists a martingale of the form
$U_t = \int_0^t Z_s dW_s$ and a RCLL predictable process
$V_t \ge \int_0^t f(s,W,Y_s,Z_s) ds$ starting at $0$ such that
$$
Y_t = Y_0+ U_t + V_t.
$$
By Lemma \ref{lemmasemi}, $Z$ is in BMO and $\N{V}_{{\cal H}^1} < \infty$.
Defining $A_t := V_t - \int_0^t f(s,W,Y_s,Z_s)ds$ shows the existence of a supersolution.

To see that the supersolution satisfies bounded comparison from
above, assume $(Y',Z',A')$ is a supersolution of the BSDE
\eqref{bsde} with terminal condition $\xi' \ge \varphi(W)$ and
driver $f' \ge f$ such that $Y'$ is bounded. Then it follows
from Theorem \ref{thm2} that $Y'_t \ge Y^n_t$ for
all $t$ and $n$. Therefore, $Y'_t \ge Y_t$ for all $t$.
\unskip\nobreak\hfill$\Box$\par\addvspace{\medskipamount}

\begin{appendix}

\setcounter{equation}{0}

\section{Appendix: The validity of Theorem 12 of Briand et al. (2002) in our setting}

The purpose of this appendix is to show that Theorem 12 of Briand et al. (2002) still holds
in the context of the proof of Theorem \ref{thm1}. Most of their arguments go through in our setup.
But where they apply Proposition 11 we use Lemma \ref{lemmaB} below.
Assume that (W1)--(W5) hold and $f$ is a driver satisfying
$$
\sup_t |f(t,0,0,0)| < \infty
$$
and
\be \label{flip}
|f(t,w_1,y_1,z_1)-f(t,w_2,y_2,z_2)| \le K \brak{\sup_{0 \le s \le t} |w_1(s) - w_2(s)| +
|y_1-y_2| + |z_1 - z_2|}
\ee
for some constant $K \in \mathbb{R}_+$. As in Theorem \ref{thm1}, $\varphi : C^d[0,T] \to \mathbb{R}$ is
assumed to be a Lipschitz-continuous function. In particular, $\varphi(W)$ is square-integrable.
Under these assumptions it follows from Pardoux and Peng (1990) that the BSDE
\eqref{bsde} has a unique solution $(Y,Z)$, and we know from Proposition
\ref{propex} that for $N$ so large that $\max_i \Delta t^N_{i} < 1/K$,
the $N$-th BS$\Delta$E has a unique solution $(Y^N,Z^N,M^N)$. We are showing
the following version of Theorem 12 of Briand et al. (2002):

\begin{theorem}
For $N \to \infty$, one has
$$
\sup_t\Big(|Y^{N}_t- Y_t|+ |\int_0^t Z^N_s dW^N_s-\int_0^t Z_s dW_s|+|M^{N}_t|\Big)
\to 0 \quad \mbox{in } L^2,
$$
and
$$
\sup_t \brak{\sum_{k=1}^d \abs{\int_0^t Z^{N,k}_s d \ang{W^N}_s - \int_0^t Z^k_s ds}^2
+ \abs{\int_0^t |Z^N_s|^2 d \ang{W^N}_s - \int_0^t |Z_s|^2 ds}}
\to 0 \quad \mbox{in } L^1.
$$
\end{theorem}

\begin{proof}
Set
$$
(Y^{\infty,0},Z^{\infty,0}) := (0,0) \quad \mbox{and} \quad
(Y^{N,0},Z^{N,0},M^{N,0}) :=(0,0,0).
$$
For $p \in \mathbb{N}$, define $(Y^{\infty,p+1},Z^{\infty,p+1})$ as follows:
$Z^{\infty,p+1}$ is the unique $d$-dimensional $({\cal F}_t)$-predictable process satisfying
\beas
\int_0^t Z^{\infty,p+1}_s dW_s &=& \mathbb{E}_{{\cal F}_t} \edg{
\varphi(W) + \int_0^T f(s,W,Y^{\infty,p}_s,Z^{\infty,p}_s)ds}\\
&& - \mathbb{E} \edg{
\varphi(W) + \int_0^T f(s,W,Y^{\infty,p}_s,Z^{\infty,p}_s)ds}
\eeas
and
$$
Y^{\infty,p+1}_t = \varphi(W) + \int_t^T
f(s,W,Y^{\infty,p}_s,Z^{\infty,p}_s)ds -\int_t^T Z^{\infty,p+1}_s dW_s.
$$
Similarly, decompose
\beas
&& \mathbb{E}_{{\cal F}^N_t} \edg{\varphi(\hat{W}^N) + \int_{(0,T]}
\hat{f}^N(s,\hat{W}^N,Y^{N,p}_{s-}, Z^{N,p}_s) d \ang{W^N}_s}\\
&-& \mathbb{E} \edg{\varphi(\hat{W}^N) + \int_{(0,T]}
\hat{f}^N(s,\hat{W}^N,Y^{N,p}_{s-}, Z^{N,p}_s) d \ang{W^N}_s}
\eeas
into a martingale of the form $\int_{(0,t]} Z^{N,p+1}_s d W^N_s$ and
a martingale $M^{N,p+1}$ orthogonal to $W^N$. Then set
\beas
Y^{N,p+1}_t &=& \varphi(\hat{W}^N) + \int_{(t,T]} \hat{f}^N(s,\hat{W}^N,Y^{N,p}_{s-}, Z^{N,p}_s)
d \ang{W^N}_s\\
&& - \int_{(t,T]} Z^{N,p+1}_s d W^N_s - (M^{N,p+1}_T - M^{N,p+1}_t).
\eeas
It is well-known from Pardoux and Peng (1990) that
$$
\E{\sup_t |Y^{\infty,p}_t - Y_t|^2 + \int_0^T |Z^{\infty,p}_s - Z_s|^2 ds}
\to 0 \quad \mbox{for } p \to \infty,
$$
and it follows as in Corollary 10 of Briand et al. (2002) that there exists
an $N_0$ such that
$$
\sup_{N \ge N_0}
\E{\sup_t |Y^{N,p}_t - Y^N_t|^2 + \int_0^T |Z^{N,p}_s - Z^N_s|^2 d\ang{W^N}_s +
|M^{N,p}_T - M^N_T|^2} \to 0 \quad \mbox{for } p \to \infty.
$$
So it is enough to show that for fixed $p$ and $N \to \infty$, one has
\be
\label{1toshow}
\sup_t \brak{|Y^{N,p}_t-Y^{\infty,p}_t| + |\int_0^t Z^{N,p}_s dW^N_s -
\int_0^t Z^{\infty,p}_s dW_s|+ |M^{N,p}_t|} \to 0 \quad
\mbox{in } L^2
\ee
and
\be \label{2toshow} \sup_t
\brak{\sum_{k=1}^d \abs{\int_0^t Z^{N,p,k}_s d \ang{W^N}_s -
\int_0^t Z^{\infty,p,k}_s ds}^2 + \abs{\int_0^t |Z^{N,p}_s|^2 d
\ang{W^N}_s - \int_0^t |Z^{\infty,p}_s|^2 ds}} \to 0
\ee
in $L^1$.
This can be proven by induction over $p$. Assume that it holds
for $p$. Then by Lemma \ref{lemmaB} below,
$$
\sup_t \Big|\int_{(0,t]}
\hat{f}^N(s,\hat{W}^N,Y^{N,p}_{s-},Z^{N,p}_s) d\ang{W^N}_s-
\int_0^t f(s,W,Y^{\infty,p}_{s},Z^{\infty,p}_s) ds \Big| \to 0
\quad \mbox{in } L^2.
$$
Moreover, one obtains as in Briand et al. (2002) that
$$
\mathbb{E}_{{\cal F}^N_t} \edg{\varphi(\hat{W}^N)}=
Y^{N,p+1}_t- \mathbb{E}_{{\cal F}^N_t} \edg{\int_{(t,T]}
\hat{f}^N(s,\hat{W}^N,Y^{N,p}_{s-}, Z^{N,p}_s) d \ang{W^N}_s}
$$
converges in ${\cal S}^2$ to
$$
\mathbb{E}_{{\cal F}_t} \edg{\varphi(\hat{W})}=Y^{\infty,p+1}_t-
\mathbb{E}_{{\cal F}_t} \edg{\int_{t}^T
f(s,W,Y^{\infty,p}_s, Z^{\infty,p}_s) d s}.
$$
So $Y^{N,p+1} \to Y^{\infty,p+1}$ in ${\cal S}^2$. Finally, since the martingale
$$\mathbb{E}_{{\cal F}^N_t} \edg{Y^{N,p+1}_T-Y^{N,p+1}_0+ \int_{(0,T]}
\hat{f}^N(s,W^N,Y^{N,p}_{s-}, Z^{N,p}_s) d \ang{W^N}_s}
=\int_{(0,t]} Z^{N,p+1}_s d W^N_s +M^{N,p+1}_t$$
converges in ${\cal S}^2$ to
$$\mathbb{E}_{{\cal F}_t} \edg{Y^{\infty,p+1}_T-Y^{\infty,p+1}_0+ \int_{0}^T
f(s,W,Y^{\infty,p}_s, Z^{\infty,p}_s) d s }=\int_{0}^t Z^{p+1}_s d W_s,$$
\eqref{1toshow}--\eqref{2toshow} follow from Theorem 5 in Briand et al. (2002).
\end{proof}

\begin{lemma} \label{lemmaB}
Fix $p \in \mathbb{N}$ and assume that
\beas
&& \sup_t |Y^{N,p}_t-Y^{\infty,p}_t|^2+\sum_{k=1}^d\abs{\int_{(0,t]} Z^{N,p,k}_s
d \ang{W^N}_s - \int_0^t Z^{\infty,p,k}_s ds}^2\\
&& + \abs{\int_{(0,t]} |Z^{N,p}_s|^2 d \ang{W^N}_s - \int_0^t |Z^{\infty,p}_s|^2 ds}
\to 0 \quad \mbox{in } L^1 \quad \mbox{for } N \to \infty.
\eeas
Then
$$
\sup_t \Big|\int_{(0,t]} \hat{f}^N(s,\hat{W}^N,Y^{N,p}_{s-},Z^{N,p}_s)
d\ang{W^N}_s- \int_0^t f(s,W,Y^{\infty,p}_{s},Z^{\infty,p}_s) ds \Big| \to 0
\quad \mbox{in } L^2
$$
for $N \to \infty$.
\end{lemma}

\begin{proof}
By definition \eqref{deffn}, one has
\beas
&&\int_{(t^N_i,t^N_{i+1}]} \hat{f}^N(s,\hat{W}^N,Y^{N,p}_{s-},Z^{N,p}_s)  d\ang{W^N}_s
=\int_{t^N_i}^{t^N_{i+1}} f(s,\hat{W}^N,Y^{N,p}_{t^N_{i}},Z^{N,p}_{t^N_{i+1}})  ds\\
&& = \int_{t^N_i}^{t^N_{i+1}} f(s,\hat{W}^N,Y^{N,p}_s,Z^{N,p}_s) ds,
\eeas
and therefore,
\bea
\label{i}
&& \sup_t \Big|\int_{(0,t]} \hat{f}^N(s,\hat{W}^N,Y^{N,p}_s,Z^{N,p}_s)
d\ang{W^N}_s - \int_{(0,t]} f(s,\hat{W}^N,Y^{N,p}_s,Z^{N,p}_s) ds \Big|^2\\
&=&
\max_i \sup_{t^N_i<t\leq t^N_{i+1}} \Big|\int_{t^N_i}^{t} f(s,\hat{W}^N,Y^{N,p}_s,Z^{N,p}_s)ds \Big|^2
\nonumber\\
&\le&
\max_i \Delta t^N_{i+1} \int_{t^N_i}^{t^N_{i+1}} |f(s,\hat{W}^N,Y^{N,p}_s,Z^{N,p}_s)|^2 ds
\nonumber\\
&\le& \nonumber
4 \max_i (\Delta t^N_{i+1})^2
\crl{\sup_t|f(t,0,0,0)|^2 + K^2 \brak{
\sup_t |W_t|^2 + |Y^{N,p}_{t^N_{i}}|^2
+|Z^{N,p}_{t^N_{i+1}}|^2}} \to 0,
\eea
in $L^1$ for $N \to \infty$, where we used \eqref{flip}
and $(a+b+c+d)^2 \le 4 (a^2+b^2+c^2+d^2)$.
Next, observe that it follows from the assumptions that for all $k=1, \dots, d$,
$$
Z^{N,p,k} \to Z^{\infty,p,k} \quad \mbox{weakly in } L^2([0,T] \times \Omega)
$$
as well as
$$
\N{Z^{N,p,k}}_{L^2([0,T] \times \Omega)} \to \N{Z^{\infty,p,k}}_{L^2([0,T] \times \Omega)}.
$$
This gives
$$
\int_0^T \big|Z^{N,p}_s - Z^{\infty,p}_s\big|^2 ds \to 0 \quad \mbox{in } L^1
\quad \mbox{as } N \to \infty,
$$
which, together with $(a+b+c)^2 \le 3 (a^2 + b^2 + c^2)$, shows that
\bea
\label{prob2}
&& \sup_t \Big| \int_0^t f(s,\hat{W}^N,Y^{N,p}_s,Z^{N,p}_s) ds
- \int_0^t f(s,W,Y^{\infty,p}_{s},Z^{\infty,p}_s)  ds\Big|^2\nonumber\\
&\le& T\int_0^T \Big|f(s,\hat{W}^N,Y^{N,p}_s,Z^{N,p}_s)
-  f(s,W,Y^{\infty,p}_s,Z^{\infty,p}_s)\Big|^2  ds \nonumber\\
&\le& 3 T^2 K^2 \sup_t \brak{|\hat{W}^N_t-W_t|^2+ |Y^{N,p}_t-Y^{\infty,p}_t|^2}
+ 3 T K^2 \int_0^T|Z^{N,p}_s-Z^{\infty,p}_{s}|^2ds \Big) \nonumber\\
&& \to 0 \quad \mbox{in } L^1 \quad \mbox{for } N \to \infty.
\eea
Combining \eqref{i} and \eqref{prob2}, one obtains
\beas
&& \sup_t \Big|\int_{(0,t]} \hat{f}^N(s,\hat{W}^N,Y^{N,p}_{s-},Z^{N,p}_s)  d\ang{W^N}_s
- \int_0^t f(s,W,Y^{\infty,p}_{s},Z^{\infty,p}_s)  ds\Big|^2\\
&\le& 2\sup_t\bigg\{\Big|\int_{(0,t]} \hat{f}^N(s,\hat{W}^N,Y^{N,p}_{s-},Z^{N,p}_s)
d\ang{W^N}_s-\int_0^t f(s,\hat{W}^N,Y^{N,p}_s,Z^{N,p}_s) ds \Big|^2\\
&& +\Big|\int_0^t f(s,\hat{W}^N,Y^{N,p}_s,Z^{N,p}_s) ds
- \int_0^t f(s,W,Y^{\infty,p}_{s},Z^{\infty,p}_s) ds \Big|^2\bigg\}
\to 0 \quad \mbox{in } L^1 \quad \mbox{as } N \to \infty.
\eeas
\end{proof}

\end{appendix}

\bigskip \bigskip \bigskip
\noindent
{\bf \Large References}\\[4mm]
%K. Bahlali (2001). Backward stochastic differential equations
%with locally Lipschitz coefficient. {\it C. R. Acad. Sci.
%Paris, Ser. I}, 331, 481--486.\\[2mm]
M.T. Barlow and P. Protter (1990). On convergence of semimartingales.
S\'eminaire de Probabilit\'es. XXIV, Lecture Notes in Math., 1426,  Springer, Berlin.\\[2mm]
P. Barrieu and N. El Karoui (2009).
Pricing, hedging and optimally designing derivatives via minimization of risk measures.
Indifference Pricing: Theory and Applications
(ed: Ren\'e Carmona), Princeton University Press.\\[2mm]
M. Ben-Artzi, P. Souplet and F. B. Weissler (2002).
The local theory for viscous Hamilton-Jacobi equations in
Lebesgue spaces. J. Math. Pures Appl. 81, 343--378.\\[2mm]
J. M. Bismut (1973). Th\'eorie probabiliste du contr\^ole des diffusions.
Mem. Amer. Math. Soc. 176.\\[2mm]
P. Briand, B. Delyon and J. M\'emin (2001).
Donsker-type theorem for BSDEs. Electron. Commun. Probab.6, 1--14.\\[2mm]
P. Briand, B. Delyon and J. M\'emin (2002). On the
robustness of backward stochastic differential equations.
Stoch. Proc. Appl. 97, 229--253.\\[2mm]
P. Briand and Y. Hu (2006). BSDEs with quadratic growth and
unbounded terminal value. Probab. Th. Rel. Fields 136, 604--618.\\[2mm]
P. Briand and Y. Hu (2008). Quadratic BSDEs with convex
generators and unbounded terminal conditions. Probab. Th. Rel. Fields 141, 543--567.\\[2mm]
P. Cheridito and M. Stadje (2009). BS$\Delta$Es and BSDEs with
non--Lipschitz drivers: properties, robustness and limit results. Preprint.\\[2mm]
F. Delbaen, Y. Hu and X. Bao (2009). Backward SDEs
with superquadratic growth. Probab. Th. Rel. Fields. forthcoming.\\[2mm]
F. Delbaen, Y. Hu and A. Richou (2011). On the uniqueness of solutions to quadratic BSDEs with convex
generators and unbounded terminal conditions.
Ann. Inst. Henri Poincar\'{e}, Prob. et Stat., 47(2), 559–-574.\\[2mm]
%N. El Karoui, S.Hamad\`ene and A.Matoussi (2009). Backward SDEs
%and applications. Indifference Pricing: Theory and Applications (ed: Ren\'e Carmona),
%Princeton University Press.\\[2mm]
N. El--Karoui, S. Peng and M.C. Quenez (1997). Backward
stochastic differential equations in finance. Math. Finance 7, 1--71.\\[2mm]
B. H. Gilding, M. Guedda and R. Kersner (2003). The Cauchy
problem for $u_t = \Delta u + |\nabla u|^q$. J. Math. Anal. Appl. 284, 733--755.\\[2mm]
N. Kazamaki (1994). Continuous exponential martingales and BMO. 
Lecture Notes in Math. 1579, Springer.\\[2mm]
M. Kobylanski (2000). Backward stochastic differential
equations and partial differential equations with quadratic
growth. Ann. Probab. 28, 259--276.\\[2mm]
%J. P. Lepeltier and J. S. Martin (1997). Backward stochastic
%differential equations with continuous coefficients.
%Statist. Probab. Lett. 34, 425--430.\\[2mm]
E. Pardoux and S. Peng (1990). Adapted solution of a backward
stochastic differential equation. Syst. Control Lett. 14, 55--61.\\[2mm]
S. Peng (1999). Monotonic limit theorem of BSDE and nonlinear
decomposition theorem of Doob-Meyers type. Prob. Rel. Fields, 113, 23--30.
%J. Popenda and J. Werbowski (1979) On the discrete analogy of Gronwall lemma, {\it Fasc.
%Math.}, 11, 143--154.
%\bibitem{Protter}
%P. Protter (2000). {\it Stochastic Integration and Differential Equations}.
%Applied Mathematics, 21, Springer, Berlin, Heidelberg, New York.
\end{document}